\documentclass[12pt, reqno]{amsart}
\usepackage{amssymb, mathrsfs}
\usepackage{latexsym}
\usepackage{amsmath}
\usepackage{amsthm}
\usepackage{amsfonts}
\usepackage{dsfont, upgreek}
\usepackage{mathtools}
\usepackage{epsfig}
\usepackage{cite}
\usepackage{amscd}
\usepackage{graphicx}
\usepackage{tikz-cd}
\usepackage{adjustbox}
\usepackage{slashed}

\usepackage{enumitem}

\usepackage[left=1.4in,right=1.4in,top=1.2in,bottom=1.2in]{geometry}

\usepackage{tikz}
\usetikzlibrary{decorations.markings}
\usetikzlibrary{arrows.meta}

\usepackage[textsize=tiny]{todonotes}

\usepackage[colorlinks=true,linkcolor=blue,citecolor=blue]{hyperref}

\usepackage[all,cmtip]{xy}
\theoremstyle{plain}
\newtheorem{theorem}{Theorem}[section]
\newtheorem{proposition}[theorem]{Proposition}
\newtheorem{lemma}[theorem]{Lemma}

\newtheorem{question}[theorem]{Question}

\theoremstyle{definition} 
\newtheorem{definition}[theorem]{Definition}

\newtheorem{geoset}[theorem]{Geometric Setup}
\newtheorem{claim}[theorem]{Claim}
\newtheorem*{claim*}{Claim}

\theoremstyle{remark} 
\newtheorem{remark}[theorem]{Remark}

\numberwithin{equation}{section}

\newcommand{\ind}{\textup{Ind}}

\newcommand{\cK}{\mathcal{K}}
\newcommand{\cL}{\mathcal{B}}
\newcommand{\cB}{\mathcal{B}}
\newcommand{\cD}{\mathcal{D}}

\newcommand{\cS}{\mathcal{S}}
\newcommand{\dist}{\mathrm{dist}}
\newcommand{\sgn}{\mathrm{\, sgn}}
\newcommand{\R}{\mathbb{R}}
\newcommand{\Z}{\mathbb{Z}}

\newcommand{\tr}{\mathrm{tr}}
\newcommand{\C}{\mathbb{C}}

\newcommand{\sC}{\partial\widetilde M\times\mathbb{R}}
\newcommand{\op}{\mathrm{op}}
\newcommand{\Ch}{\mathrm{Ch}}
\newcommand{\STR}{\mathrm{STR}}
\newcommand{\ch}{\mathrm{ch}}
\newcommand{\cyc}{\mathrm{c}}
\newcommand{\per}{\mathrm{per}}
\newcommand{\cl}{\mathbb{C}\ell}

\newcommand{\hotimes}{\mathbin{\widehat{\otimes}}}
\newcommand{\spin}{\mathrm{Spin}}
\DeclareMathOperator{\End}{\mathrm{End}}

\mathchardef\mhyphen="2D

\let\latexchi\chi
\makeatletter
\renewcommand\chi{\@ifnextchar_\sub@chi\latexchi}
\newcommand{\sub@chi}[2]{
	\@ifnextchar^{\subsup@chi{#2}}{\latexchi^{}_{#2}}%
}
\newcommand{\subsup@chi}[3]{
	\latexchi_{#1}^{#3}%
}
\makeatother

\begin{document}
	
\title{$l^1$-higher index, $l^1$-higher rho invariant and cyclic cohomology}

\author{Jinmin Wang}
\address[Jinmin Wang]{Department of Mathematics, Texas A\&M University}
\email{jinmin@tamu.edu}
\thanks{The first author is partially supported by NSF 1952693.}
\author{Zhizhang Xie}
\address[Zhizhang Xie]{ Department of Mathematics, Texas A\&M University }
\email{xie@math.tamu.edu}
\thanks{The second author is partially supported by NSF 1800737 and 1952693.}
\author{Guoliang Yu}
\address[Guoliang Yu]{ Department of
	Mathematics, Texas A\&M University}
\email{guoliangyu@math.tamu.edu}
\thanks{The third author is partially supported by NSF  2000082 and the Simons Fellows Program.}

\maketitle
\begin{abstract}
In this paper, we study $l^1$-higher index theory and its pairing with cyclic cohomology for both closed manifolds and compact manifolds with boundary. We first give a sufficient geometric condition for the vanishing of the $l^1$-higher indices of Dirac-type operators on closed manifolds. This leads us to define an $l^1$-version of higher rho invariants. We prove a product formula for these $l^1$-higher rho invariants.  A main novelty of  our product formula is that it works in the general Banach algebra setting, in particular, the $l^1$-setting. 

On compact spin manifolds with boundary, we also give a  sufficient geometric condition for Dirac operators to have well-defined $l^1$-higher indices. More precisely, we show that,  on a compact spin manifold $M$ with boundary equipped with a Riemannian metric which has product structure near the boundary,  if the scalar curvature on the boundary is sufficiently large,  then the $l^1$-higher index of its Dirac operator $D_M$ is well-defined and lies in the $K$-theory of the $l^1$-algebra of the fundamental group.  As an immediate corollary, we see that if the Bost conjecture holds for the fundamental group of $M$, then the $C^\ast$-algebraic higher index of $D_M$ lies in the image of the Baum-Connes assembly map. 

By pairing the above $K$-theoretic $l^1$-index results with cyclic cocycles, we prove an $l^1$-version of the higher Atiyah-Patodi-Singer index theorem for manifolds with boundary. A key ingredient of its proof is the product formula for $l^1$-higher rho invariants mentioned above. 
\end{abstract}
\section{Introduction}

The higher index for elliptic differential operators on closed manifolds is far-reaching generalization of the classical Fredholm index for Fredholm operators. Usually, the higher index an elliptic operators on a closed manifold takes values in the $K$-theory of the reduced or maximal group $C^*$-algebra of the fundamental group of the manifold.

Suppose $(X, g)$ is a closed spin manifold equipped with a Riemannian metric $g$. If the scalar curvature $\kappa$ of $g$ is positive, then it follows from the Lichnerowicz formula that the Dirac operator $D$ on $X$ is invertible, which implies that the higher index $\ind^\Gamma(D)$ of the Dirac operator $D$  vanishes 
 in $K_*(C^*_r(\Gamma))$ and $K_*(C^*_{\max}(\Gamma))$. Here $\Gamma = \pi_1(X)$ is the fundamental group of $X$, and  $C^*_r(\Gamma)$ (resp. $C^*_{\max}(\Gamma)$) is  the reduced (resp. maximal) group $C^*$-algebra of $\Gamma$.  This vanishing result has some interesting  geometric consequence. For example, it implies  the non-existence of positive scalar metrics on aspherical manifolds whose fundamental groups satisfy the strong Novikov conjecture, cf. \cite{Yu,Yucoarse}.

 A standard construction of higher indices shows that the higher index of an elliptic operator on a closed manifold can always be represented by $K$-theory elements with finite propagation (cf. Definition \ref{def:fp,lc,pl}). In particular, each such representative also lies in the $K$-theory $K_*(l^1(\Gamma))$ of the $l^1$-algebra of $\Gamma$. Let us denote the $l^1$-higher index class of an elliptic operator $D$ by $\ind_1^\Gamma(D)$.  We will review the construction of these $l^1$-higher indices  in Section \ref{sec:higherindexclosed}. Unlike the $C^\ast$-algebraic setting, due to the lack of positivity in $l^1$-algebras, it  remains an open question if the Lichnerowicz-type vanishing result still holds in the $l^1$-setting. 
 \begin{question}
 	Assume that $(X, g)$ is a closed spin Riemannian manifold equipped with a Riemannian metric $g$ whose scalar curvature is positive. Then  does the $l^1$-higher index $\ind_1^\Gamma\Gamma(D)$ of the Dirac operator $D$ on $X$ vanish in $K_*(l^1(\Gamma))$?
 \end{question}

Our first main theorem of the current paper is a partial answer to the above question. Namely, we show that, if the scalar curvature of $g$ is sufficiently large in an appropriate sense, then the $l^1$-higher index $\ind_1^\Gamma\Gamma(D)$ vanishes in $K_*(l^1(\Gamma))$. Before we state the theorem, let us fix some notation. 
Let $\widetilde X$ be the universal cover of $X$ and $\mathcal F$ be a fundamental domain of the $\Gamma$-action on $\widetilde X$. Fix a finite symmetric generating set $S$ of $\Gamma$. Let $\ell$ be the length function on $\Gamma$ induced by $S$. We define 
\begin{equation}\label{eq:comparison}
\tau=\liminf_{\ell(\gamma)\to\infty}\sup_{x\in\mathcal F}\frac{\dist(x,\gamma x)}{\ell(\gamma)}
\end{equation}
and
\begin{equation}\label{eq:rate}
K_\Gamma=\inf\{K:\exists C>0 \text{ s.t. }\#\{\gamma\in\Gamma:\ell(\gamma)\leqslant n\}\leqslant Ce^{Kn} \}.
\end{equation} 
\begin{theorem}\label{theorem:main1}
	Let $(X,g)$ be an $n$-dimensional closed spin Riemannian manifold. Denote the fundamental group $\pi_1(X)$ of $X$ by $\Gamma$. If the scalar curvature $\kappa$ of $g$ satisfies that
	\begin{equation}\label{eq:sufflarge}
	\inf_{x\in X}\kappa(x)>\frac{16K_\Gamma^2}{\tau^2},
	\end{equation}
	then the $l^1$-higher index $\ind_1^\Gamma(D)$ of the Dirac operator $D$ on $X$ vanishes in $K_n(l^1(\Gamma))$.
\end{theorem}
The numbers $K_\Gamma$ and $\tau$ may depend on the choice of generating set, but Theorem \ref{theorem:main1} holds as long as the inequality $\eqref{eq:sufflarge}$ is satisfied for one choice of generating set. In particular, if the group $\Gamma$ has sub-exponential growth, i.e., $K_\Gamma=0$, then Theorem \ref{theorem:main1} holds as long as the scalar curvature is positive everywhere.

Recall that in the $C^\ast$-algebraic setting, the vanishing of  the higher index of an elliptic operator together with a specific trivialization  naturally induces a $C^\ast$-algebraic secondary invariant called higher rho invariant. These secondary invariants have various interesting applications in geometry and topology.  For example, a positive scalar curvature metric on a closed spin manifold naturally defines a higher rho invariant associated to the given metric.  The higher rho invariants for two positive scalar curvature metrics are the same if they are homotopic along a path of positive scalar curvature metric. Consequently, we can use higher rho invariants to distinguish different connected components of the space of positive scalar curvature metrics on a given closed spin manifold.  In the $l^1$-setting, under the condition that the scalar curvature is sufficiently large \eqref{eq:sufflarge}, we introduce an $l^1$-higher rho invariant, denoted by $\rho_1(g)$,  that lies in some geometric Banach algebra $\cB_{L,0}(\widetilde X)^\Gamma$, cf. Definition \ref{def:local,obst}. This $l^1$-higher rho invariant also remains unchanged for a continuous path of metrics, provided each metric has sufficiently large scalar curvature. For $C^\ast$-algebraic higher rho invariant, there is a product formula (cf. \cite{Xiepos} \cite{Zeidlerproduct}), which is important for geometric applications. In this paper,  we  prove a product formula for $l^1$-higher rho invariants. More precisely, we have the following theorem. Denote by $C^*_L(\R)$ the localization algebra of the real line and $\ind_L(D_\R)\in K_1(C^*_L(\R))$ the local index of the Dirac operator on $\R$, cf. \textup{\cite{Yulocalization}}. 
\begin{theorem}\label{thm:introproduct}
	Assume that $(X,g)$ is an $m$-dimensional closed spin Riemannian manifold such that the scalar curvature function $\kappa$ of $g$ satisfies the inequality in line \eqref{eq:sufflarge}. Then under the $K$-theory product map
	\begin{equation}
	_-\otimes _-\colon  K_*(\cB_{L,0}(\widetilde X)^\Gamma)\otimes K_1(C^*_L(\R))\to K_{*+1}(\cB_{L,0}(\widetilde X\times\R)^\Gamma),
	\end{equation}
	 we have
	\begin{equation}
	\rho_{1}(g)\otimes \ind_L(D_\R)=\rho_{1}(g+dt^2)\in K_{m+1}(\cB_{L,0}(\widetilde X\times\R)^\Gamma),
	\end{equation}
where $\cB_{L,0}(\widetilde X)^\Gamma$ is a certain geometric Banach algebra \textup{(}cf. Definition $\ref{def:local,obst}$\textup{)}, and $g+dt^2$ is the product metric on $X\times \mathbb R$ induced by $g$ and the standard Euclidean metric on $\mathbb R$.  
\end{theorem}

When applied to the $C^\ast$-algebraic setting, our approach also proves the product formula for $C^\ast$-algebraic higher rho invariants. However, a main novelty of our approach is that it applies to the general Banach algebra setting, while the existing approaches in the $C^\ast$-algebra setting  (cf. \cite{Xiepos} \cite{Zeidlerproduct}) do not seem  to generalize to the Banach algebra setting.

Now we turn to the case of manifolds with boundary. Given a compact spin manifold $M$ with boundary, if $M$ is equipped with a Riemannian metric which has product structure near the boundary and the scalar curvature is positive on the boundary, then the Dirac operator $D$ on $M$ defines a higher index $\ind^\Gamma(D)$ which lies in the $K$-theory of the reduced group $C^\ast$-algebra of the fundamental group $\Gamma = \pi_1(M)$ of $M$, cf. Section \ref{sec:1}. It is an open question  whether such a higher index lies in the image of the Baum-Connes assembly map. Let us briefly recall the definition of the Baum-Connes assembly map.  Denote by $\underline{E}\Gamma$  the classifying space of $\Gamma$ for proper actions, and $K_*^\Gamma(\underline{E}\Gamma)$ its equivariant $K$-homology group. Each element of $K_*^\Gamma(\underline{E}\Gamma)$ can be represented by a $\Gamma$-equivariant Dirac operator (possibly twisted by an auxiliary vector bundle)  on a complete $spin^c$ manifold (without boundary) that is equipped with an isometric, proper and cocompact $\Gamma$-action.   The Baum-Connes assembly map
\[ \mu\colon  K_*^\Gamma(\underline{E}\Gamma)\longrightarrow K_*(C_r^*(\Gamma)) \]
is defined by mapping each twisted $\Gamma$-equivariant Dirac operator to its associated higher index.   

The Baum--Connes conjecture claims that the assembly map
is an isomorphism. The Baum-Connes conjecture has been verified for a large class of groups, including groups with Haagerup property \cite{MR1821144} and hyperbolic groups \cite{L1,MR1914618}. In general, the conjecture is still wide open. In fact, the following more special question is still wide open. 
\begin{question}\label{q:main}
Given  a compact spin manifold with boundary whose Riemannian metric  has product structure and positive scalar curvature near the boundary, does the higher index of its Dirac operator  lie in the image of the Baum-Connes assembly map? 
\end{question} 
 A positive answer to the above question can be used to compute certain secondary index theoretic invariants, such as delocalized eta invariants and higher rho invariants  associated to positive scalar curvature metrics on the boundary of a spin manifold,  cf. \cite{WXY,Xie}.

Although we are mainly interested in Question $\ref{q:main}$ in the $C^\ast$-algebraic setting, the main tools we use in this paper for tackling such a question  are in fact from the $l^1$-setting. Let us briefly explain how the $l^1$-setting and the $C^\ast$-algebraic setting are related.  Recall that the algebra of $l^1$-functions on $\Gamma$ is a Banach algebra and a dense subalgebra of $C^*_r(\Gamma)$. There is an analogue of the Baum-Connes conjecture in the $l^1$-setting, called the Bost conjecture,   which states that
\begin{equation}\label{eq:bost}
	\mu_1 \colon  K_*^\Gamma(\underline{E}\Gamma)\longrightarrow K_*(l^1(\Gamma))
\end{equation}
is an isomorphism. Here we have used that fact the standard construction of higher indices of twisted Dirac operators on  $spin^c$ $\Gamma$-manifolds (without boundary) implies that these higher indices  can be represented by elements with finite propagation (cf. Definition \ref{def:fp,lc,pl}), hence in particular lie in $K_*(l^1(\Gamma))$. Consequently,  the Baum-Connes assembly map  $\mu\colon  K_*^\Gamma(\underline{E}\Gamma)\to K_*(C_r^*(\Gamma))$ factors through the Bost assembly map 	$\mu_1 \colon  K_*^\Gamma(\underline{E}\Gamma)\to K_*(l^1(\Gamma))$ as follows: 
\[  K_*^\Gamma(\underline{E}\Gamma)\to  K_*(l^1(\Gamma)) \to K_*(C_r^*(\Gamma)), \]
where the second homomorphism is induced by the natural inclusion \[ l^1(\Gamma) \hookrightarrow C^*_r(\Gamma). \] 

One advantage of the $l^1$-setting is that the Bost conjecture has been proved for a larger class of groups, which in particular includes all lattices in reductive Lie groups  \cite{L1}, while  it is  not yet known whether or not the Baum-Connes conjecture holds for $\textup{SL}_3(\mathbb Z)$. On the other hand, the Baum-Connes conjecture is more directly applicable to geometry and topology. In some sense, a main strategy of this paper is to work in a geometric setup where both conjectures intersect, that is, a geometric setup where methods from both the $C^\ast$-algebraic setting  and the  $l^1$-setting apply. In particular, our next theorem states that if the scalar curvature on the boundary of a spin manifold is sufficiently large (in the sense of the inequality from line $\eqref{eq:sufflarge}$), then the higher index of the associated Dirac operator, which is a prior an element in $K_\ast(C^\ast_r(\Gamma))$, actually lies in $K_\ast(l^1(\Gamma))$.

Before we state the theorem, let us fix some notation. 
Let $\widetilde M$ be the universal cover of $M$ and $\mathcal F_\partial$ be a fundamental domain of the $\Gamma$-action on $\partial\widetilde M$. Fix a finite symmetric generating set $S$ of $\Gamma$. Let $\ell$ be the length function on $\Gamma$ induced by $S$. We define 
\begin{equation}\label{eq:tauboundary}
	\tau_\partial=\liminf_{\ell(\gamma)\to\infty}\sup_{x\in\mathcal F_\partial}\frac{\dist(x,\gamma x)}{\ell(\gamma)}
\end{equation}
and $K_\Gamma$ as in line \eqref{eq:rate}.
\begin{theorem}\label{theorem:main}
	Let $M$ be an $n$-dimensional compact spin manifold with boundary. Suppose $M$ is equipped with a Riemannian metric $g$ that has product structure and positive scalar curvature near the boundary. Denote  the fundamental group $\pi_1(M)$ of $M$ by $\Gamma$. If the scalar curvature $\kappa$ on $\partial M$ satisfies that
	\begin{equation}\label{eq:sufflarg}
\inf_{x\in\partial M}\kappa(x)>\frac{16K_\Gamma^2}{\tau^2},
	\end{equation}
	then the $C^*$-algebraic higher index $\ind^\Gamma(D)$ of the Dirac operator $D$  on $M$ admits a natural preimage $\ind^\Gamma_1(D)$ in $K_n(l^1(\Gamma))$ under the homomorphism 
		 \[ K_n(l^1(\Gamma)) \to K_n(C_r^*(\Gamma)) \]
		 induced by the inclusion $l^1(\Gamma) \hookrightarrow C^*_r(\Gamma)$. 
\end{theorem}

As before, the numbers $K_\Gamma$ and $\tau_\partial$ may depend on the generating set, but Theorem \ref{theorem:main1} holds as long as the inequality $\eqref{eq:sufflarg}$ is satisfied for one choice of generating set. In particular, if the group $\Gamma$ has sub-exponential growth, i.e., $K_\Gamma=0$, then Theorem \ref{theorem:main} holds as long as the scalar curvature on the boundary is positive.

As an application of Theorem $\ref{theorem:main}$,  we have the following theorem, which gives a partial positive answer to Question $\ref{q:main}$ in the geometric setup of Theorem $\ref{theorem:main}$. 
\begin{theorem}
	Let $M$ be an $n$-dimensional compact spin manifold with boundary $\partial M$. Suppose $M$ is equipped with a Riemannian metric $g$ that has product structure near the boundary such that the scalar curvature $\kappa$ of $g$ on $\partial M$ satisfies that
	$$\inf_{x\in\partial M}\kappa(x)>\frac{16K_\Gamma^2}{\tau^2_\partial}.$$
If the Bost conjecture holds for $\Gamma = \pi_1(M)$, then the $C^*$-algebraic higher index $\ind^\Gamma(D)$ of  the Dirac operator $D$ on $M$ lies in the image of the Baum--Connes assembly map.
\end{theorem}

Furthermore, by applying the techniques we develop for proving Theorem $\ref{theorem:main}$, we shall prove an $l^1$-version of the higher  Atiyah-Patodi-Singer index formula. Roughly speaking, this $l^1$-higher  Atiyah-Patodi-Singer index formula states that, for a given compact spin manifold $M$ with boundary such that its Dirac operator is invertible on the boundary,  the pairing between the higher index of the Dirac operator and a cyclic cohomology class $\varphi$ of $\mathbb C\Gamma$ is equal to the sum of the integral of a local expression on $M$ and a higher eta invariant from the boundary. In particular, if $\varphi$ is a delocalized cyclic $n$-cocycle of $\mathbb C\Gamma$ (cf. Definition $\ref{def:delocal}$),
then the local term in such a pairing vanishes, thus the index pairing is equal to a higher eta invariant from the boundary in this case.  

To state our $l^1$-higher Atiyah-Patodi-Singer index theorem, we shall need the following notion of exponential growth rate for cyclic cocycles of $\Gamma$. We say a cyclic $n$-cocycle $\varphi$ of $\Gamma$  has at most exponential growth with respect to a given length function $\ell$ on $\Gamma$ if there exist $K>0$ and $C>0$ such that
 $$|\varphi(\gamma_0,\gamma_1,\cdots,\gamma_n)|\leqslant Ce^{K (\ell(\gamma_0)+\ell(\gamma_1)+\cdots+\ell(\gamma_n))},\ \forall \gamma_i\in\Gamma.$$
In this case, the infimum of such constants $K$ is called the exponential growth rate of $\varphi$ and will be denoted by $K_\varphi$ from now on.

\begin{theorem}\label{theorem:aps}
	Let $M$ be a compact spin manifold with boundary $\partial M$ and $D$ the Dirac operator on $M$. Suppose $\varphi$ is a delocalized cyclic cocycle of $\Gamma = \pi_1(M)$ that has exponential growth rate $K_\varphi$. If $M$ is equipped with a Riemannian metric $g$ that has product structure near the boundary such that the scalar curvature $\kappa$ of $g$ on $\partial M$ satisfies that
	\begin{equation}\label{eq:sufflarge2}
\inf_{x\in\partial M}\kappa(x)>\frac{16(K_\Gamma+K_\varphi)^2}{\tau^2_\partial} 
	\end{equation}
	then we have
	\begin{equation}\label{eq:aps}
\ch_\varphi(\ind^\Gamma(D))=-\frac 1 2\eta_\varphi(D_\partial),
	\end{equation}
	where $\ch_\varphi(\ind^\Gamma(D))$ is the pairing between $\varphi$ and the Connes-Chern character of the higher index $\ind^\Gamma(D)$, and $\eta_\varphi(D_\partial)$ is the higher eta invariant \textup{(}with respect to $\varphi$\textup{)} of the Dirac operator $D_\partial$ on $\partial M$. 
\end{theorem}
Let us recall the definitions of  $\ch_\varphi(\ind^\Gamma(D))$ and  $\eta_\varphi(D_\partial)$ (cf.  \cite{CWXY}). We will only give the formulas in the case where $\dim(M)$ is even. The odd dimensional case is similar. For simplicity, let us write $ p = \ind^\Gamma(D)$, a representative with finite support on $\Gamma$.   If $\varphi$ is a delocalized cyclic $2m$-cocycle, then we have 
\begin{equation}\label{eq:conneschern}
\ch_\varphi(\ind^\Gamma(D))\coloneqq \frac{(2m)!}{m!} \varphi\# \tr(p^{\otimes 2m+1})
\end{equation} 
and 
\begin{equation}\label{eq:highereta}
	\eta_\varphi(\widetilde D_\partial)
\coloneqq \frac{m!}{\pi i}
\int_0^\infty \varphi\#\tr(\dot u_s u_s^{-1}\otimes( ((u_s-1)\otimes (u_s^{-1}-1)) )^{\otimes m})ds,
\end{equation}
where $\dot{u}_s$ is the derivative of $u_s$ with respect to $s\in (0, \infty)$ and $u_s$ is defined as follows. Let us denote the universal covering space of $M$ by $\widetilde M$. The boundary of  $\widetilde M$ is a $\pi_1(M)$-covering space of $\partial M$.  Let $\widetilde D_\partial$ be the lift of the Dirac operator $D_\partial$ on $\partial M$ to this covering space of $\partial M$. We define 
$$u_s \coloneqq e^{2\pi i \cdot  f(s^{-1}\widetilde D_\partial)}  \textup{ with }  f(x)=\frac{1}{\sqrt\pi}\int_{-\infty}^{x} e^{-y^2}dy.
 $$
It was proved in \cite[Theorem 3.24]{CWXY} that the integral formula for 
the  higher eta invariant $\eta_\varphi(D_\partial)$  in line $\eqref{eq:highereta}$ converges absolutely, provided that the scalar curvature of $\partial M$ is sufficiently large, i.e., satisfying the condition in line $\eqref{eq:sufflarge2}$. In proving Theorem $\ref{theorem:aps}$, we will also show that the formula in $\eqref{eq:conneschern}$ (which is a summation of infinitely many terms) absolutely converges, provided that the condition in line $\eqref{eq:sufflarge2}$ holds. A key ingredient for the proof of Theorem \ref{theorem:aps} is Theorem \ref{thm:introproduct}---the product formula for $l^1$-higher rho invariants.

Theorem $\ref{theorem:aps}$ improves the higher Atiyah-Patodi-Singer index formula in  \cite[Theorem 3.36]{CLWY}. The index formula $\eqref{eq:aps}$ in Theorem $\ref{theorem:aps}$ can be viewed as the pairing of the $K$-theoretic higher Atiyah-Patodi-Singer index formula (as in \cite{Piazzarho,Xiepos}) with cyclic cohomology. The main difficulties for proving $\eqref{eq:aps}$  are to justify the convergence of the formulas for $\ch_\varphi(\ind^\Gamma(D))$ and $\eta_\varphi(D_\partial)$, and to establish the equality while working with K-theory of Banach algebras (instead of $C^\ast$-algebras). Furthermore, although we have only stated Theorem $\ref{theorem:aps}$ for delocalized cyclic cocycles, the same result actually holds for all cyclic cocycles with at most exponential growth except there is an extra local term on the right hand side (cf. Section $\ref{sec:alter}$).  

The paper is organized as follows. In Section \ref{sec:indexclosed}, we discuss the $l^1$-higher index theory for closed manifolds. We prove a vanishing theorem for the $l^1$-higher index of Dirac operator when the given scalar curvature is sufficient large. Furthermore, we introduce an $l^1$-version of higher rho invariants and prove a product formula of these secondary invariants in the $l^1$-setting. In Section \ref{sec:1}, we review the construction of $C^\ast$-algebraic higher index for  Dirac operators on compact spin manifolds whose boundary has positive scalar curvature. In Section \ref{sec:2}, we construct   $l^1$-higher indices  of   Dirac operators on compact spin manifolds,  provided the scalar curvature on the boundary is sufficiently large. We prove a generalized version of  Theorem $\ref{theorem:main}$ (cf. Theorem $\ref{thm:main}$ for the details). In Section Section $\ref{sec:APS}$, we apply the techniques in Section  \ref{sec:2} to prove  an $l^1$-higher Atiyah-Patodi-Singer index theorem (Theorem $\ref{theorem:aps}$).

\section{$l^1$-higher index theory for closed spin manifold}\label{sec:indexclosed}
In this section, we review the construction of the $l^1$-higher index of the Dirac operator and prove a vanishing theorem for the index. As an application, we define a higher rho invariant and prove a product formula.
\subsection{Geometric $C^*$-algebras and their Banach analogues}
In this section, we review the construction of  Roe algebras and their analogue in the $l^1$-setting. 

Let $X$ be a closed spin Riemannian manifold and $S(X)$ the spinor bundle over $X$.
Denote by $\Gamma$ the fundamental group of $X$. Let $\widetilde X$ be the universal cover of $X$ and $S(\widetilde X)$ be the lift of the spinor bundle. 

Set $H=L^2(S(\widetilde X))$, the Hilbert space of all square-integrable sections of $S(\widetilde X)$. Let $B(H)^\Gamma$ be the collection of all $\Gamma$-equivariant bounded operator on $H$. We choose a precompact fundamental domain $\mathcal F$ of the $\Gamma$-action on $\widetilde X$ and denote by $\psi$ the characteristic function of $\mathcal F$.

Let  $\|\cdot\|_\op$ be the operator norm on $B(H)$. For any $T\in  B(H)^\Gamma$ and $\gamma\in\Gamma$, we define
\[ T_{\gamma}\coloneqq \psi\circ  T \circ \gamma\psi. \] 
\begin{definition}\label{def:B^Gamma}
	We define $\cL(H)^{\Gamma}$ to be the following subspace of linear operators
	$$\cL(H)^{\Gamma}\coloneqq \{T\in B(H)^\Gamma:\sum_{\gamma\in \Gamma}\|T_\gamma\|_{\op} <\infty \}$$
	equipped with the norm
	\begin{equation}\label{eq:onenorm}  
	\|T\|_1=\sum_{\gamma\in \Gamma}\|T_\gamma\|_{\op}.
	\end{equation}
\end{definition}

It is easy to verify that $(\cB(H)^{\Gamma},\|\cdot \|_1)$ is a Banach $\ast$-algebra. Moreover, the definition of $\cB(H)^{\Gamma}$  is independent of the choice of the fundamental domain $\mathcal F$.

\begin{definition}\label{def:fp,lc,pl}\ 
		\begin{enumerate}
		\item We say that $T\in B(H)^\Gamma$ has \emph{finite propagation} if there exists $d>0$ such that $\chi_A\circ T \circ \chi_B=0$ for any two Borel sets $A$ and $B$ with $\dist(A,B)>d$. Here for example $\chi_A$ is the characteristic function of $A$.  The infimum of such $d$ is called the \emph{propagation} of $T$.
		\item  $T\in B(H)^\Gamma$ is called \emph{locally compact} if $\chi_A \circ T$ and $T\circ\chi_A$ are compact for any precompact Borel set $A$.
	\end{enumerate}
\end{definition}

\begin{definition}\label{def:Roe}
		Let $\C(\widetilde X)^\Gamma$ be the collection of operators $T\in B(H)^\Gamma$ such that $T$ has finite propagation and is locally compact. Let $C^*(\widetilde X)^\Gamma$ (resp. $\cL(\widetilde X)^\Gamma$)  be the completions of $\C(\widetilde X)^\Gamma$ with respect to the operator norm $\|\cdot\|_\op$ (resp. the $l^1$-norm $\|\cdot \|_1$ given in line $\eqref{eq:onenorm}$).
\end{definition}

The map
$T\mapsto \sum_{\gamma\in\Gamma} T_{\gamma}\gamma$
 induces isomorphisms
$$C^*(\widetilde X)^\Gamma\cong \cK\otimes C^*_r(\Gamma)\text{ and }\cL(\widetilde X)^\Gamma\cong \cK\otimes l^1(\Gamma)$$
where $\cK$ is the algebra of compact operators and the norm on $\cK\otimes l^1(\Gamma)$ is given by the $l^1$-norm
\[   \big\|\sum_{\gamma\in \Gamma} a_\gamma \gamma\big\| := \sum_{\gamma\in \Gamma}\|a_\gamma\|_{op},~ a_\gamma\in\cK,\]
which coincides with the maximal tensor product norm.

It is a well-known fact that $K$-theory of $C^\ast$-algebras is stable, that is, 
\[ K_\ast(A\otimes \cK) = K_\ast(A) \] for any $C^\ast$-algebra $A$. The following lemma shows that $K$-theory is also stable for $l^1$-algebras.\footnote{It remains an open question whether $K$-theory is stable for general Banach algebras. More precisely, for any Banach algebra $B$, is there a suitable topological tensor product $B\otimes \cK$ such that $K_*(B\otimes \cK) \cong K_*(B)$? }

%
\begin{proposition}\label{lemma:stable}
	Let $p$ be a rank one projection in $\cK$. The following map
	$$\iota\colon  l^1(\Gamma)\to \cK\otimes l^1(\Gamma),\ a\mapsto p\otimes a$$
	induces an isomorphism at the level of $K$-theory.
\end{proposition}
\begin{proof}
	Given an element $\sum_{\gamma\in \Gamma}a_\gamma\gamma\in \cK\otimes l^1(\Gamma)$, for $
	\varepsilon>0$, there exist a finite subset $F\subset \Gamma$ and finite dimensional matrices $a'_\gamma \in M_n(\mathbb C)$ for each $\gamma\in F$ such that $$\big\|\sum_{\gamma\in \Gamma}a_\gamma\gamma-\sum_{\gamma\in F}a_\gamma'\gamma\big\|_1
	=\sum_{\gamma\in F}\|a_\gamma-a_\gamma'\|_{op}+\sum_{\gamma\notin F}\|a_\gamma\|_{op}
	<\varepsilon.$$ Thus $\cK\otimes l^1(\Gamma)$ is the direct limit of the following direct system:
	$$l^1(\Gamma)\xrightarrow{\iota_1} M_1(\mathbb C)\otimes l^1(\Gamma)\xrightarrow{\iota_2}  M_2(\mathbb C)\otimes l^1(\Gamma) \xrightarrow{\iota_3} \cdots$$
	where the map 	
	\[ \iota_n \colon M_n(\mathbb C)\otimes l^1(\Gamma)\to M_{n+1}(\mathbb C)\otimes l^1(\Gamma), \]
	is given by 
	\[ \sum_{\gamma\in \Gamma}b_\gamma\gamma\mapsto \sum_{\gamma\in \Gamma}\begin{pmatrix}b_\gamma&\\&0\end{pmatrix}\gamma,
	\]
	and the norm on $M_n(\mathbb C)\otimes l^1(\Gamma)$ is given by
	$$\|\sum_{\gamma\in \Gamma}b_\gamma\gamma\|_1=\sum_{\gamma\in \Gamma}\|b_\gamma\|_\op.$$
	Clearly, the map $\iota_n$ induces an isomorphism at the level of $K$-theory, hence follows the lemma. 	
	
\end{proof}

Now we fix a symmetric generating set of $\Gamma$ and denote by $\ell$ the corresponding length function. We will need the following weighted $l^1$-completions of $\mathbb C\Gamma$. 
\begin{definition} \label{def:konenorm}		Suppose $K$ is a non-negative real number.
	\begin{enumerate}
		\item Let $l^1_K(\Gamma)$ be the completion of $\mathbb C\Gamma$ with respect to the following norm
		\begin{equation}
			\|\sum_{\gamma\in\Gamma}a_\gamma\gamma\|_{1,K}:=\sum_{\gamma\in\Gamma} e^{K\ell(\gamma)}\|a_\gamma\|,
		\end{equation}
		for all finite sums  $\sum_{\gamma\in\Gamma}a_\gamma\gamma \in\mathbb C\Gamma.$
		\item Similarly, define $\cL(\widetilde X)_K^\Gamma$ to be the completion of $\mathbb C(\widetilde X)^\Gamma$ with respect to the following norm
		$$\|T\|_{1,K}:=\sum_{\gamma\in \Gamma} e^{K\ell(\gamma)}\|T_{\gamma}\|_\op,
		$$
		for all $T = \sum_{\gamma\in\Gamma}T_\gamma\gamma \in \mathbb C(\widetilde X)^\Gamma.$
	\end{enumerate}
\end{definition}

It is clear that $l^1_K(\Gamma)$ and $\cL(\widetilde X)_K^\Gamma$ are Banach algebras and $\cL(\widetilde X)_K^\Gamma$ is isomorphic to the maximal tensor product $l^1_K(\Gamma)\otimes\cK$.
\subsection{$l^1$-higher index of Dirac operators}\label{sec:higherindexclosed}
Now we recall the definition of the higher index in the $l^1$-setting. 

\begin{definition}\label{def:normalizingfunc}
	A continuous function $F\colon \mathbb R\to [-1, 1]$ is called a normalizing function if 
	\begin{enumerate}
		\item $F$ is an odd function, that is, $F(-t) = -F(t)$,
		\item $\lim_{x\to\pm\infty} F(x)=\pm 1,$
		\item the Fourier transform of $F'$ is supported on $[-N_F,N_F]$ for some $N_F>0$.
	\end{enumerate}
\end{definition}
Note that, in this case, the Fourier transform of $F$ is also supported on $[-N_F,N_F]$ as a tempered distribution. 
Throughout the paper, we use the following choice of formula for the Fourier transform
\begin{equation}\label{eq:Fourier}
\hat{f}(\xi)=\int f(x)e^{-ix\xi}dx,
\end{equation}
and its inverse Fourier transform is given by 
\begin{equation}\label{eq:fourier}
f(x)=\frac{1}{2\pi}\int \hat f(\xi)e^{ix\xi}dx.
\end{equation}

Let $\widetilde D$ be the associated Dirac operator on the universal cover $\widetilde X$. By the Fourier inverse transform formula \eqref{eq:fourier}, we see that the operator $F(\widetilde D)$ has propagation no more than $N_F$, since the wave operator $e^{i\xi\widetilde D}$ has propagation $\leqslant |\xi|$.
Therefore, we have 
$$\|F(\widetilde D)\|_{1,K}<\infty$$
for any $K\geqslant 0$.

When $X$ is even dimensional, the spinor bundle $S$ on $X$ admits a natural $\Z_2$-grading and the Dirac operator $D$ is an odd operator.
As $F$ is an odd function, $F(\widetilde D)$ is an odd operator with respect to the  $\mathbb Z_2$-grading on the spinor bundle $\widetilde S$ of $\widetilde X$, that is,
\begin{equation}
F(\widetilde D)=\begin{pmatrix}
0&F(\widetilde D)^-\\F(\widetilde D)^+&0
\end{pmatrix}.
\end{equation}
Let us write $U = F(\widetilde D)^+$ and $V = F(\widetilde D)^-$, and define  an invertible element
\begin{equation}
W_{F(\widetilde D)}\coloneqq \begin{pmatrix}1& U\\ 0&1\end{pmatrix}
\begin{pmatrix}1& 0\\ -V&1\end{pmatrix}
\begin{pmatrix}1& U\\ 0&1\end{pmatrix}
\begin{pmatrix}0& -1\\ 1&0\end{pmatrix}
\end{equation}
This allows us to define the following idempotent 
\begin{equation}\label{eq:idem}
\begin{split}
p_{F(\widetilde D)}&=W_{F(\widetilde D)}\begin{pmatrix}
1&0\\0&0
\end{pmatrix} W_{F(\widetilde D)}^{-1}\\
&=\begin{pmatrix}
1-(1-UV)^2&(2-UV)U(1-VU)\\V(1-UV)&(1-VU)^2
\end{pmatrix}.
\end{split}
\end{equation}
It is easy to see that  $p_{F(\widetilde D)}-\begin{psmallmatrix}1&0\\0&0\end{psmallmatrix}$ is locally compact hence lies in $\cB(\widetilde X)^\Gamma_K$.

\begin{definition}\label{def:1Kindex-even}
	If $\dim X$ is even, then the $l^1$-higher index of $\widetilde D$ is defined to be
	$$\ind^\Gamma_{1,K}(D)\coloneqq[p_{F(\widetilde D)}]-[\begin{psmallmatrix}1&0\\0&0\end{psmallmatrix}]\in K_0(\cB(\widetilde X)^\Gamma_K).$$
\end{definition}

When $X$ is odd-dimensional, we see that
$e^{2\pi i\frac{F(\widetilde D)+1}{2}}$
is locally compact and lies in the unitalization of $\cB(\widetilde X)^\Gamma_K$, as $\|\cdot\|_{1,K}$ is an algebraic norm.
\begin{definition}\label{def:1Kindex-odd}
	If $\dim X$ is odd, then the $l^1$-higher index of $\widetilde D$ is defined to be
$$\ind^\Gamma_{1,K}( D)\coloneqq [e^{2\pi i\frac{F(\widetilde D)+1}{2}}]\in K_1(\cB(\widetilde X)^\Gamma_K).$$
\end{definition}

Obviously, the definition of the $l^1$-higher index is independent of the choice of the normalizing function. In fact, any function $F$ in Definition \ref{def:normalizingfunc} with the condition (3) replaced by that 
$$\|F(\widetilde D)\|_{1,K}<\infty$$
also defines the same index class in $K_*(\cB(\widetilde X)^\Gamma_K)$.
\subsection{$l^1$-norm inequalities}
In this subsection, we prove some $l^1$-norm estimates of various operators. 

We first fix some notations. Choose a finite symmetric generating set $S$ of $\Gamma$. Let $\ell$ be the length function on $\Gamma$ induced by $S$. We define
\begin{equation}\label{eq:KGamma}
K_\Gamma=\inf\{K:\exists C>0 \text{ s.t. }\#\{\gamma\in\Gamma:\ell(\gamma)\leqslant n\}\leqslant Ce^{Kn} \}
\end{equation}
and
\begin{equation}\label{eq:tau0}
\tau=\liminf_{\ell(\gamma)\to\infty}\sup_{x\in\mathcal F}\frac{\dist(x,\gamma x)}{\ell(\gamma)},
\end{equation}
where $\mathcal F$ is a fundamental domain of the $\Gamma$-action on
 $\widetilde X$.



For convenience, instead of the definitions of $K_\Gamma$ and $\tau$ in line \eqref{eq:KGamma} and \eqref{eq:tau0}, we assume that
\begin{equation}\label{eq:KGammaconvenience}
\#\{\gamma\in\Gamma:\ell(\gamma)\leqslant n \}\leqslant Ce^{K_\Gamma n},~\forall n\geqslant 0
\end{equation}
and
\begin{equation}\label{eq:tau}
\dist(x,\gamma x)>\tau\ell(\gamma)-C_0,~\forall\gamma\in \Gamma,~\forall x\in \mathcal F
\end{equation}
with some $C>0$ and $C_0>0$.
We recall that $\psi$ is the characteristic function of $\mathcal F$. We denote the diameter of $\mathcal F$ by $\mathrm{diam}\mathcal F$. 
\begin{lemma}\label{lemma:Fourier}	
	Given $f\in C_0(\R)$, suppose its Fourier transform $\hat f\in L^1(\R)$. Then for any $\mu>1$, we have
	$$\|f(\widetilde D)_\gamma\|_\op\leqslant \frac{1}{2\pi}\int_{|\xi|>\frac{\tau\ell(\gamma)}{\mu}} |\hat f(\xi)|d\xi $$
	for all $\gamma\in\Gamma$ satisfying
	\begin{equation}\label{eq:lengthbd}
	\ell(\gamma)>\mathfrak N_\mu\coloneqq \frac{\sqrt{\mu}(C_0+\mathrm{diam}\mathcal F)}{\tau(\sqrt\mu-1)},
	\end{equation}
	where $f(\widetilde D)_\gamma = \psi \circ f(\widetilde D) \circ \gamma \psi$.
\end{lemma}
\begin{proof}
	Fix $\mu>1$. For each $\gamma\in \Gamma$ such that $\ell(\gamma) >\mathfrak N_\mu$, let $\chi$ be a smooth function on $\mathbb R$ that vanishes on the interval $[-\frac{\tau\ell(\gamma)}{\mu},\frac{\tau\ell(\gamma)}{\mu}]$ and equals $1$ on $(-\infty, -\frac{\tau\ell(\gamma)}{\sqrt\mu})$ and $(\frac{\tau\ell(\gamma)}{\sqrt\mu},+\infty)$.
	
	By the Fourier inverse transform formula, we have
	$$f(\widetilde D)=\frac{1}{2\pi}\int \hat f(\xi)e^{i\xi \widetilde D}d\xi.$$
	Let us define
	$$g(\widetilde D)=\frac{1}{2\pi}\int \chi(\xi)\hat f(\xi)e^{i\xi \widetilde D}d\xi.$$
	Since the wave operator $e^{i\xi \widetilde D_{\cyc}}$ has propagation $\leqslant |\xi|$, it follows that $f(\widetilde D_{\cyc})-g(\widetilde D_{\cyc})$ has propagation  $\leqslant \frac{\tau\ell(\gamma)}{\sqrt\mu}$. 	
	Therefore, from line \eqref{eq:tau}, we have $f(\widetilde D)_\gamma-g(\widetilde D)_\gamma=0$ in this case. We conclude that 
	\begin{align*}
	\|f(\widetilde D)_\gamma\|_\op=\|g(\widetilde D)_\gamma \|_\op\leqslant \|g(\widetilde D)\|_\op \leqslant \frac{1}{2\pi}\int_{|\xi|>\frac{\tau\ell(\gamma)}{\mu}} |\hat f(\xi)|d\xi.
	\end{align*}
\end{proof}


\begin{definition}
	Let $f$ be a Schwartz function on $\mathbb R$. We say $f$ has Gaussian decay with rate $C$ if there exists $A>0$ such that
	$$|f(x)|\leqslant A e^{-Cx^2}$$
	for all $x\in \mathbb R$. 
\end{definition}

For example, the Gaussian function $e^{-x^2}$ as well as its Fourier transform has Gaussian decay.
\begin{lemma}\label{lemma:largespectralgap}
	Given a Schwartz function $f$ on $\mathbb R$, suppose both $f$ and its Fourier transform $\hat f$  have Gaussian decay with rates $C$ and $\widehat C$ respectively. Let $\sigma$ be the infimum of the spectrum of $|\widetilde D|$. If there exists a constant $K\geqslant 0$ such that 
	$$\sigma>\frac{K_\Gamma + K}{\tau\sqrt{C\widehat C}},$$
	then there exist $\varepsilon>0$ and $\lambda >0$ such that
	$$\|f(t\widetilde D)\|_{1, K}\leqslant \lambda \cdot e^{-\varepsilon t^2}$$
	where $\|\cdot\|_{1, K}$ is the weighted $l^1$-norm given in Definition $\ref{def:konenorm}$.
\end{lemma}
\begin{proof}
	By assumption, there exist positive constants $A_0$ and $A_1$ such that
	$$|f(x)|\leqslant A_0 e^{-C x^2}\text{ and }|\hat f(\xi)|\leqslant A_1 e^{-\widehat C\xi^2}.$$
	Applying Lemma \ref{lemma:Fourier}, it follows that for any $\mu>1$,  there exists $A_2 >0$ such that
	$$\|f(t\widetilde D)_\gamma\|_\op\leqslant A_2 e^{-\widehat C\frac{\tau^2\ell(\gamma)^2}{t^2\mu^2}}$$
	for all $\gamma\in \Gamma$ with $\ell(\gamma)\geqslant \mathfrak N_\mu$, where $\mathfrak N_\mu$ is the number given in line $\eqref{eq:lengthbd}$. On the other hand, since $\sigma$ is the spectral gap of $\widetilde D$ at zero, we have 
	$$\|f(t\widetilde D)_\gamma\|_\op\leqslant \|f(t\widetilde D)\|_\op \leqslant A_0e^{-C t^2\sigma^2}.$$
	By assumption that $\sigma>\frac{K_\Gamma + K}{\tau\sqrt{C\widehat C}}$, we have 
	$$\frac{C\sigma^2}{K_\Gamma+ K}>\frac{(K_\Gamma + K)\mu^2 }{\widehat C \tau^2}$$
	for some $\mu>1$. Fix such a $\mu$ for the rest of the proof. There exists a positive number  $s$ such that 
	$$\frac{C\sigma^2}{(K_\Gamma+ K)s}>1 > \frac{(K_\Gamma + K)s }{\widehat C \tau^2 \mu^{-2} s^2}.$$It follows that there exists $\varepsilon > 0$ such that 
	\[ 
	C\sigma^2-(K_\Gamma+ K)s>\varepsilon>0 \textup{ and } 
	\frac{\widehat C\tau^2s^2}{\mu^2}-(K_\Gamma+ K)s>\varepsilon>0.
	\]
	We conclude  that 
	\begin{align*}
	\|f(tD)\|_{1, K}=&\sum_{\ell(g)\leqslant st^2 +\mathfrak N_\mu} e^{K\ell(\gamma)} \|f(tD_{\cyc})_\gamma\|_\op+\sum_{\ell(g)>st^2 +\mathfrak N_\mu} e^{K\ell(\gamma)}\|f(tD_{\cyc})_\gamma\|_\op\\
	\leqslant&A_0 e^{-Ct^2\sigma^2}e^{(K_\Gamma +K) (st^2 + \mathfrak N_\mu)}+
	A_2\sum_{n>st^2 + \mathfrak N_\mu}e^{-\widehat C\frac{\tau^2n^2}{t^2\mu^2}}e^{(K_\Gamma +K) \cdot n}\\
	\leqslant& A_0 e^{(K_\Gamma +K) \mathfrak N_\mu} e^{-(C\sigma^2-K_\Gamma s - Ks)t^2}+A_2\sum_{n>st^2+ \mathfrak N_\mu}e^{-(\widehat C\frac{\tau^2 s}{\mu^2}-K_\Gamma -K)\cdot n } \\
	\leqslant &A_3 e^{-\varepsilon t^2}+A_4 e^{-\varepsilon  t^2}
	\end{align*}
	for some $A_3 >0$ and $A_4>0$. The lemma follows by setting $\lambda = A_3 + A_4$.  
\end{proof}
\subsection{A vanishing theorem for $l^1$-higher index}
In this subsection, we show that the $l^1$-higher index of an elliptic differential operator $\widetilde D$ vanishes if the spectral gap of $\widetilde D$ at zero is sufficiently large.

\begin{theorem}\label{thm:vanishing}
	For any $K\geqslant 0$, if
	\begin{equation}\label{eq:|D|>}
	|\widetilde D|>\frac{2(K_\Gamma+K)}{\tau},
	\end{equation}
	then $\ind^\Gamma_{1,K}(D)=0$ in $K_*(\cB(\widetilde X)^\Gamma_K)$.
\end{theorem}

For example, by the Lichnerowicz formula, line \eqref{eq:|D|>} holds when the scalar curvature function $\kappa(x)$ on $M$ satisfying that there exists $K\geqslant 0$ such that
$$
\kappa(x)>\frac{16(K_\Gamma+K)^2}{\tau^2}\textup{ for all } x\in X.
$$
In particular, if $K=0$ and the group $\Gamma$ has sub-exponential growth (i.e. $K_\Gamma=0$), then line \eqref{eq:|D|>} holds as long as $\widetilde D$ is invertible.
\begin{proof}[Proof of Theorem $\ref{thm:vanishing}$]
	Let $\sgn$ be the sign function
\[ \sgn(x) = \begin{cases}
1 & \textup{ if } x>0, \\
0 & \textup{ if } x=0, \\
-1 & \textup{ if } x<0.
\end{cases} \] 
	We will prove in Lemma \ref{lemma:G_tfinite} and Lemma \ref{lemma:G-sgn} that under the assumption \eqref{eq:|D|>}, $\|\sgn(\widetilde D)\|_{1, K}$ is finite, and  there exists a family of normalizing function 
	$\{G_t\}$ parameterized by $t\in[1,+\infty)$ such that
	$\|G_t(\widetilde D)\|_{1,K}$ is uniformly bounded and
	$$\lim_{t\to\infty}\|G_t(\widetilde D)-\sgn(\widetilde D)\|_{1,K}=0.$$
	Since $\sgn(\widetilde D)^2=1$, this shows that $\ind^\Gamma_{1,K}( D)$ is trivial.
\end{proof}
We consider the normalizing function  
\begin{equation}\label{eq:G(x)}
G(x)=\frac{2}{\sqrt\pi}\int_{-\infty}^x e^{-y^2}dy-1
\end{equation}
and set $G_t(x)=G(tx)$. 
\begin{lemma}\label{lemma:G_tfinite}
	For each $t>0$, $\|G_t(\widetilde D)\|_{1, K}$ is finite.
\end{lemma}
\begin{proof}
	Without loss of generality, let us assume $t=1$. 
	
	The distributional Fourier transform of $G$ is given by
	\begin{equation}
	\widehat G(\xi)=\frac{2i}{\xi}e^{-\frac{\xi^2}{4}}-2\pi\delta(\xi),
	\end{equation}
	where $\delta$ is the delta function concentrated at the origin.
	Let $\chi$ be a cut-off function such that $\chi(\xi)=1$ if $|\xi|<1$, $\chi(\xi)=0$ if $|\xi|>2$ and $|\chi(\xi)|\leqslant 1$  for all $\xi\in \mathbb R$. We write $G$ as a sum of two functions  $G = G^{(1)}+G^{(2)}$,  where the Fourier transforms of $G^{(1)}$ and $G^{(2)}$ are $\widehat{G}^{(1)}=\chi\cdot\widehat G$ and $\widehat{G}^{(2)}=(1-\chi)\cdot\widehat G$ respectively.
	
	As $\widehat{G}^{(1)}$ has compact support, it follows from the Fourier inverse transform formula (cf. line $\eqref{eq:fourier}$) that $G^{(1)}(\widetilde D)$ has finite propagation. Therefore $\|G^{(1)}(\widetilde D)\|_{1, K}$ is finite.
	On the other hand, $\widehat{G}^{(2)}$ is a Schwartz function with Gaussian decay. By Lemma \ref{lemma:Fourier}, there exists $\varepsilon>0$ such that
	$$\|G^{(2)}(\widetilde D)_\gamma\|_\op\leqslant  e^{-\varepsilon \ell(\gamma)^2}$$
	as long as $\ell(\gamma)$ is sufficiently large. It follows that  $\|G^{(2)}(\widetilde D)\|_{1, K}$ is also finite. This finishes the proof. 
\end{proof}
\begin{remark}\label{rk:bounded}
	The proof of Lemma \ref{lemma:G_tfinite} above does not require the invertibility of  $\widetilde D$.
\end{remark}

Note that, under the assumption that $\widetilde D$ is invertible, it is clear that 
\begin{equation*}
G_t(\widetilde D)\to \sgn(\widetilde D) \textup{ in the operator norm, as } t\to \infty.
\end{equation*}
The following lemma shows that the above convergence still holds with respect to the $\|\cdot\|_{1, K}$ norm, provided the spectral gap of $\widetilde D$ at zero is sufficiently large.

\begin{lemma}\label{lemma:G-sgn}
	If $|\widetilde D|>\frac{2(K_\Gamma+K)}{\tau}$, then $\|\sgn(\widetilde D)\|_{1, K}$ is finite. Furthermore, there exist $C>0$ and $\varepsilon>0$ such that for any $t\geqslant 1$,
	\begin{equation}
	\|G_t(\widetilde D)-\sgn(\widetilde D)\|_{1, K}\leqslant  Ce^{-\varepsilon t^2}.
	\end{equation}
\end{lemma}
\begin{proof}
	It follows from the proof of Lemma \ref{lemma:G_tfinite} that $G_t(\widetilde D)$ is a differentiable in $t$ with respect to the $\|\cdot\|_{1,K}$ norm. We have 
	\begin{equation}\label{eq:G_t'}
	\frac{d}{dt}G_t(\widetilde D_{\cyc})=\frac{2}{\sqrt\pi}\widetilde D_{\cyc} e^{-t^2\widetilde D_{\cyc}^2}=\frac{1}{t}h(t\widetilde D_{\cyc}),
	\end{equation}
	where $h(x)=\frac{2}{\sqrt\pi}xe^{-x^2}$. Note that $h$ and $\hat h$ have Gaussian decay with rates $1$ and $1/4$ respectively. By Lemma \ref{lemma:largespectralgap}, there exist positive constants $C_1$ and $\varepsilon>0$ such that
	\begin{equation}
	\|h(t\widetilde D)\|<C_1e^{-\varepsilon t^2}.
	\end{equation}
	We conclude that  
	\begin{align*}
	\|G_t(\widetilde D)-G_T(\widetilde D)\|_{1,K}&\leqslant \int_{t}^{T}\|\frac{d}{dt}G_t(\widetilde D)\|_{1,K}\leqslant 
	\int_t^T\frac{1}{s}C_1e^{-\varepsilon s^2}ds\\
	&\leqslant Ce^{-\varepsilon t^2}
	\end{align*}
	for some fixed $C>0$ and for all $t<T$. The lemma follows by letting $T$ go to infinity.
\end{proof}
\subsection{$l^1$-higher rho invariant}\label{subsec:higher rho}

In this subsection, we introduce $l^1$-higher rho invariants for Dirac operators on complete spin manifolds whose scalar curvature is sufficiently large.

\begin{definition}\label{def:local,obst}
	We denote by $\cL_{L}(\widetilde X)^\Gamma_K$ the completion of the collection of maps $f\colon[0,\infty)\to \C(\widetilde X)^\Gamma$ that satisfy
	\begin{enumerate}
		\item $f$ is uniformly bounded and uniformly continuous in $t$ with respect to $\|\cdot \|_{1,K}$, cf., Definition \ref{def:konenorm},
		\item the propagation of $f$ goes to zero as $t$ goes to infinity,
	\end{enumerate}
	with respect to the norm $\|\cdot \|_{1,K}$ given by 
	$$\|f\|_{1,K}:=\sup_{s\geqslant 0}\|f(s)\|_{1,K}.$$
	Furthermore, we define
	$$\cL_{L,0}(\widetilde X)^\Gamma_K=\{f\in \cL_{L}(\widetilde X)^\Gamma_K: f(0)=0 \}.$$
\end{definition}

Let us first prove the following technical lemma, which gives a rather general sufficient condition for producing elements in the Banach localization algebra 
$\cB_{ L}(\widetilde X)^\Gamma_K$.

\begin{lemma}\label{lemma:local_appro}
	Let $f\in C_0(\R)$ be a continuous function on $\mathbb R$ vanishing at infinity.  Let $\chi$ be a smooth cut-off function on $\mathbb R$ such that $\chi(\xi)=0$ on $[-1,1]$, $\chi=1$ for $|\xi|\geqslant 2$ and $|\chi(\xi)|\leqslant 1$ for all $\xi\in \mathbb R$. Denote $\chi_N(\xi)=\chi(\xi/N)$ for $N>0$. If there exist positive numbers $N$, $C_1$, and $C_2$ such that $\chi_N\cdot f\in C_0(\R)$ and 
	$$|\chi_N(\xi)\hat f(\xi)|\leqslant C_1e^{-C_2|\xi|},$$
	then there exists $s_0\geqslant 0$ such that 
	the path
	$$\textstyle h(s) = f\big(\frac{\widetilde D}{s+s_0}\big) \textup{ with } s\in [0, \infty)$$
	defines an element in $\cB_{ L}(\widetilde X)^\Gamma_K$. Moreover, the norm of $h$ only depends on $N,C_1,C_2$ and the geometric data of $X$.
\end{lemma}
\begin{proof}
	For any $\alpha>N$, we define
	$$g_\alpha(x)=\frac{1}{2\pi}\int \chi(\alpha^{-1}\xi)\hat f(\xi)e^{ix\xi}d\xi.$$
	Set  $f_s(x)=f(x/s)$ and $g_{\alpha,s}(x)=g_\alpha(x/s)$.
	Since
	$$|g_\alpha(x)|\leqslant \frac{1}{2\pi}\int_{|\xi|>\alpha}C_1e^{-C_2|\xi|}d\xi
	\leqslant \frac{C_1}{\pi C_2}e^{-C_2\alpha},
	$$
	we have
	\begin{equation}\label{eq:gas1}
	\|g_{\alpha,s}(\widetilde D)_\gamma\|_\op\leqslant \|g_{\alpha,s}(\widetilde D)\|_\op\leqslant \frac{C_1}{\pi C_2}e^{-C_2\alpha}.
	\end{equation}
	
	On the other hand, the Fourier transform of $g_{\alpha,s}$ is given by
	$$\hat{g}_{\alpha,s}(\xi)=s\cdot \hat{g}_{\alpha}(s\xi)=\chi(\alpha^{-1}s\xi)\hat f(s\xi).
	$$
	If $s\geqslant 1$, then $|\hat{g}_{\alpha,s}(\xi)|\leqslant C_1 e^{-C_2s|\xi|}$.
	By Lemma \ref{lemma:Fourier}, there exist $C_1'>0$ and $C_2'>0$ independent of $\alpha,s$ such that
	\begin{equation}\label{eq:gas2}
	\|g_{\alpha,s}(\widetilde D)_\gamma\|_\op\leqslant C_1' e^{-C_2's\ell(\gamma)}.
	\end{equation}
	By combining line \eqref{eq:gas1} and \eqref{eq:gas2} together, we have for any $s\geqslant 1$, 
	$$ \|g_{\alpha,s}(\widetilde D)_\gamma\|_\op\leqslant Ce^{-\frac{C_2\alpha}{2} -\frac{C_2's}{2}\ell(\gamma)},
	$$
	where $C=\sqrt{C_1C_1'}$.
	Therefore, if we choose $s_0$ large enough such that \[ e^{(K +K_\Gamma)-s_0C_2'/2}\leqslant \frac{1}{2},\] then for any $s\geqslant s_0$,
	\begin{align*}
	\|g_{\alpha,s}(\widetilde D_X)\|_{1, K}&\leqslant \sum_{\gamma\in \Gamma} e^{K\ell(\gamma)} Ce^{-\frac{C_2\alpha}{2} -\frac{C_2's}{2}\ell(\gamma)}
	\leqslant Ce^{-\frac{C_2\alpha}{2}}\sum_{n=0}^\infty e^{-\frac{C_2's}{2}n+(K_\Gamma + K) n}
	\leqslant 2Ce^{-\frac{C_2\alpha}{2}}.
	\end{align*}
	Hence if $s\geqslant s_0$, then
	$$\|f_{s}(\widetilde D)-(f_s(\widetilde D)-g_{\alpha,s}(\widetilde D))\|_{1, K} = \|g_{\alpha,s}(\widetilde D)\|_{1, K}
	\leqslant 2Ce^{-\frac{C_2\alpha}{2}}.
	$$
	Now  observe that the operator  $(f_s(\widetilde D)-g_{\alpha,s}(\widetilde D))$ has propagation $\leqslant 2\alpha/s$. In particular, the path 
	\[ s \mapsto T_\alpha(s) = f_s(\widetilde D)-g_{\alpha,s}(\widetilde D) \] 
	with $s\in [0, \infty)$ is a uniformly bounded and uniformly continuous with respect to the $\|\cdot\|_{1, K}$-norm and the propagation of $T_\alpha(s)$ goes to $0$,  as $s\to \infty$.  We conclude that  the element 
	\[ \textstyle h(s) = f\big(\frac{\widetilde D}{s+s_0}\big)  \] 
	lies in $\cB_{K}(\widetilde X)^\Gamma$ for each $s\in [0, \infty)$, and the path $h$  can be approximated uniformly by paths such as $T_\alpha$ by letting $\alpha \to \infty$. This shows that $h$ defines an element in $\cB_{ L}(\widetilde X)^\Gamma_K$, hence finishes the proof.
\end{proof}

The following theorem gives a large collection of functions $f$ in $C_0(\mathbb R)$ that satisfy the condition  of Lemma \ref{lemma:local_appro}.
\begin{theorem}[{\cite[Theorem IX.14]{MR0493420}}]\label{thm:analytic-cont}
	Assume that $f\in C_0(\R)$ admits an analytic continuation to  $\{z\in \mathbb C:|\mathrm{Im}z|<a\}$ for some $a>0$. Let $\varphi(x) = f(x+ ib)$ for some fixed $b\in \mathbb R$. Suppose there is $M>0$ such that $\varphi\in L^1(\R)$ and $\|\varphi\|\leqslant M$ for all $|b|<2a/3$. Then $\hat f$ is a bounded continuous function on $\R$. Furthermore, there exists $C>0$ that only depends on $M$ and $a$ such that
	$$|\hat f(\xi)|\leqslant Ce^{-\frac{a|\xi|}{3}}$$
	for all $\xi\in \mathbb R$.
\end{theorem}

Given a closed spin Riemannian manifold $X$ with metric $g$, we assume that the scalar curvature function $\kappa(x)$ of $g$ satisfying that there exists $K\geqslant 0$ such that
\begin{equation}\label{eq:k>}
\kappa(x)>\frac{16(K_\Gamma+K)^2}{\tau^2}\textup{ for all } x\in X,
\end{equation}
 where $K_\Gamma$ and $\tau$ are given in line \eqref{eq:KGammaconvenience} and \eqref{eq:tau} respectively.

Assume $\dim X$ is odd for the moment.  Let us define
\begin{equation}\label{eq:u_s}
u_s=\begin{cases}
\exp(2\pi i F(s^{-1}\widetilde D))&s>0,\\
1&s=0,
\end{cases}
\end{equation}
where $F(x)\coloneqq\frac{G(x)+1}{2}$ and  $G$ is the function defined in line \eqref{eq:G(x)}.

\begin{lemma}\label{lemma:oddhighrho}
	The path $\{u_s\}_{s\in [0, \infty)}$ given in line \eqref{eq:u_s} is an invertible element in $(\cL_{L,0}(\widetilde X)^\Gamma_K)^+$, the unitization of $\cL_{ L,0}(\widetilde X)^\Gamma_K$.
\end{lemma}
\begin{proof}
	Let $\sgn$ be the sign function and $H$ the Heaviside step function $ \frac{\sgn+1}{2}$. By the invertibility of $\widetilde D$, the element $H(\widetilde D)$ is an idempotent. Therefore, we have
	\[ \exp(2\pi i\cdot H(\widetilde D_X)) = 1. \] 
	Hence
	$$u_s
	=\exp(2\pi iF(s^{-1}\widetilde D)-2\pi iH(s^{-1}\widetilde D)).$$
	The function  
	$ 2\pi i(F-H)$ satisfies  the condition in Lemma \ref{lemma:local_appro}. It follows that there exists $s_0 \geqslant 0$ such that the path 
	\[ \{ 2\pi iF(s^{-1}\widetilde D_X)-2\pi iH(s^{-1}\widetilde D_X)\}_{s\in [s_0, \infty)} \]
	defines an element in $\cL_{ L}(\widetilde X)^\Gamma_K$.
	On the other hand, it follows from Lemma \ref{lemma:G_tfinite} and Lemma \ref{lemma:G-sgn} that $u_s$ is a continuous map from $[0,s_0]$ to $\cB(\widetilde X)^\Gamma_K$. Therefore the path 
	$\{u_{s}\}_{s\in [0, \infty)}$ defines an element in $(\cL_{L,0}(\widetilde X)^\Gamma_K)^+$. Similarly, the path
	\begin{equation}
	u_s^{-1}=\begin{cases}
	\exp(-2\pi i F(s^{-1}\widetilde D))&s>0,\\
	1&s=0,
	\end{cases}
	\end{equation}
	lies in $(\cL_{L,0}(\widetilde X)^\Gamma_K)^+$ and is the inverse of $\{u_{s}\}_{s\in [0, \infty)}$.
\end{proof}

Now we are ready to define the $l^1$-higher rho invariant. 
\begin{definition}\label{def:higherrho}
	Let $(X,g)$ be a closed spin Riemannian manifold and $\widetilde X$ the universal cover of $X$. Suppose the scalar curvature  of $g$ satisfies the inequality in line \eqref{eq:k>}. 
	\begin{enumerate}
		\item If $\dim X$ is odd, then we define the $l^1$-higher rho invariant $\rho_{1,K}(g)$ to be the class $[u]\in K_1(\cL_{L,0}(\widetilde X)^\Gamma_K)$, where $u = \{u_s\}_{s\in [0, \infty)}$.
		\item If $\dim X$ is even, then we define the $l^1$-higher rho invariant $\rho_{1,K}(g)$ to be the class $[ p]-[\begin{psmallmatrix}
		1&0\\0&0
		\end{psmallmatrix}]\in K_0(\cL_{L,0}(\widetilde X)^\Gamma_K)$, where $p = \{p_s\}_{s\in [0,\infty)}$ and 
		\begin{equation}\label{eq:p_s}
		p_s=\begin{cases}
		p_{G(s^{-1}\widetilde D)}&s>0,\\
		& \\
		\begin{psmallmatrix}
		1&0\\0&0
		\end{psmallmatrix}&s=0,
		\end{cases}
		\end{equation} 
		with the formula for $p_{G(s^{-1}\widetilde D)}$ given in line \eqref{eq:idem}. 
	\end{enumerate} 
\end{definition}

Observe that, if there exits a continuous path of metrics  each of which satisfies  the condition $\eqref{eq:k>}$, then these metrics define the same $l^1$-higher rho invariant. 

\subsection{A product formula for $l^1$-higher rho invariant}
In this subsection, we prove a product formula for the $l^1$-higher rho invariant. Our method of proof is inspired by Higson--Roe \cite{Higson1,Higson2} and Weinberger--Xie--Yu \cite{Weinberger:2016dq}.

As in the previous subsection, assume  that $(X,g)$ is a closed spin Riemannian manifold with scalar curvature function $\kappa(x)$ of $g$ satisfying that there exists $K\geqslant 0$ such that
$$
\kappa(x)>\frac{16(K_\Gamma+K)^2}{\tau^2}\textup{ for all } x\in X,
$$
 where $K_\Gamma$ and $\tau$ are given in line \eqref{eq:KGammaconvenience} and \eqref{eq:tau} respectively. 
We consider the Riemannian manifold $X\times \R$ with the product metric $dt^2+g$, whose scalar curvature is still uniformly bounded below by a positive number. Similar to Definition \ref{def:Roe}, Definition \ref{def:local,obst} and Definition \ref{def:higherrho}, we define the geometric Banach algebras
\footnote{Since the cocompactness fails for the $\Gamma$-action on $\widetilde X\times\R$, the Banach algebra $\cB(\widetilde X\times\R)^\Gamma_K$ is no longer isomorphic to $\cK\otimes l^1_K(\Gamma)$.}
 $\cB(\widetilde X\times\R)^\Gamma_K$, $\cB_L(\widetilde X\times\R)^\Gamma_K$ and $\cB_{L,0}(\widetilde X\times\R)^\Gamma_K$ as well as the higher rho invariant 
$$\rho_{1,K}(dt^2+g)\in K_*(\cB_{L,0}(\widetilde X\times\R)^\Gamma_K).$$
Note that, although the $\Gamma$-action on $\widetilde X\times\R$ is no longer cocompact, these definitions still make sense due to the following observations.
\begin{enumerate}
	\item The $l^1$-norms of the operators in  $\C(\widetilde X\times\R)^\Gamma$ are defined by using the fundamental domain $\mathcal F\times\R$, where $\mathcal F$ is a precompact fundamental domain for the $\Gamma$-action on $\widetilde X$, where $\Gamma = \pi_1(\widetilde X)$. Compare with Definition \ref{def:B^Gamma} and Definition \ref{def:konenorm}.
	\item With respect to the given length function $\ell$ on $\Gamma$, we still have 
	\begin{equation}
	\dist((x,t),(\gamma x,t'))>\tau\ell(\gamma)-C_0,~\forall\gamma\in \Gamma,~\forall x\in \mathcal F,~\forall t,t'\in\R
	\end{equation}
	with the same constants $C_0$ and $\tau$ as in line \eqref{eq:tau}.
\end{enumerate}

Let  $C^*_L(\R)$ be the $C^\ast$-algebraic  localization algebra of the metric space $\R$, and   $\ind_L(D_\R)\in K_1(C^*_L(\R))$ the local index class of the Dirac operator $D_\R=i\frac{d}{dx}$, cf. \cite{Yulocalization}. Then there is a natural product map
$$\cB_{L,0}(\widetilde X)^\Gamma_K\otimes_{\max{}} C^*_L(\R)\to \cB_{L,0}(\widetilde X\times\R)^\Gamma_K,$$
which induces a homomorphism\footnote{In fact, a standard Elienberg swindle argument shows that the map $$_-\otimes \ind_L(D_\R)\colon K_*(\cB_{L,0}(\widetilde X)^\Gamma_K)\to K_{*+1}(\cB_{L,0}(\widetilde X\times\R)^\Gamma_K)$$ is an isomorphism.} at the level of $K$-theory
\begin{equation}\label{eq:K-theoryprod}
_-\otimes \ind_L(D_\R)\colon K_*(\cB_{L,0}(\widetilde X)^\Gamma_K)\to K_{*+1}(\cB_{L,0}(\widetilde X\times\R)^\Gamma_K).
\end{equation}
\begin{theorem}\label{thm:product}
	Assume that $(X,g)$ is an $m$-dimensional closed spin Riemannian manifold with scalar curvature function $\kappa(x)$ of $g$ satisfying that there exists ${K\geqslant 0}$ such that
	$$
	\kappa(x)>\frac{16(K_\Gamma+K)}{\tau}\textup{ for all } x\in X.
	$$ Then under the product map in line \eqref{eq:K-theoryprod}, we have
	$$\rho_{1,K}(g)\otimes \ind_L(D_\R)=\rho_{1,K}(g+dt^2)\in K_{m+1}(\cB_{L,0}(\widetilde X\times\R)^\Gamma_K).$$
\end{theorem}

Before we prove the theorem, let us introduce a  different (but equivalent) representative  of the $l^1$-higher rho invariant $\rho_{1,K}(g)$. First, we assume that $\dim X$ is odd. 
Set
\begin{equation}\label{eq:v_s}
v_s=\begin{cases}
(\widetilde D+si)(\widetilde D-si)^{-1} & \textup{ if } s>0,\\ 
1 & \textup{ if } s=0.\\
\end{cases}
\end{equation}
\begin{lemma}\label{lm:con-est}
	For any $K\geqslant 0$, if $\dim X$ is odd and the scalar curvature $\kappa$ on $X$ satisfies that
	$$\inf_{x\in X}\kappa(x)>\frac{16(K_\Gamma+K)^2}{\tau^2},$$
	then $ v = \{v_s\}_{s\in[0,\infty)}$ in line \eqref{eq:v_s} is an invertible element in 
	$(\cB_{L,0}(\widetilde X)^\Gamma_{K})^+$. Furthermore, we have
	$$[v]=\rho_{1,K}(g)\in K_1(\cB_{L,0}(\widetilde X)^\Gamma_K).$$
\end{lemma}
\begin{proof}
	By Lemma \ref{lemma:largespectralgap},  we have
	$$\|e^{-y(s^2\widetilde D^2+1)}\|_{1,K}\leqslant Ce^{-y(\varepsilon s^2+1)}$$
	for some $C>0$ and $\varepsilon>0$.	It follows that the integral 
	\begin{equation}\label{eq:1/D^2+1}
	\frac{1}{s^2\widetilde D^2+1}  = -\int_0^\infty e^{-y(s^2\widetilde D^2+1)}dy
	\end{equation}
	absolutely converges with respect to the $\|\cdot\|_{1,K}$-norm, hence 
	\[ \frac{1}{s^2\widetilde D^2+1}\in \cL(\widetilde X)^\Gamma_{K} \]  and 
	\begin{equation}
	\Big\|\frac{1}{s^2\widetilde D^2+1}\Big\|_{1,K}\leqslant \frac{C}{s^2\varepsilon^2+1}
	\end{equation}
	for all $s>0$. Similarly, the integral 
	\begin{equation}
	\frac{1}{|\widetilde D|}=\frac{1}{\sqrt\pi}\int_{-\infty}^{+\infty}e^{-s^2\widetilde D^2}ds
	\end{equation}
	absolutely converges with respect to the $\|\cdot\|_{1,K}$-norm. It follows that 
	$$\frac{1}{|\widetilde D|}\in\cL(\widetilde X)^\Gamma_{K},$$ which implies that 
	$$\frac{1}{\widetilde D}=\frac{\sgn(\widetilde D)}{|\widetilde D|}\in\cL(\widetilde X)^\Gamma_{K}$$
	since by Lemma \ref{lemma:G-sgn} we have $\sgn(\widetilde D)\in \cL(\widetilde X)^\Gamma_{K}$.
	
	Let $\sigma(\widetilde D^{-1})$ be the spectrum of $\widetilde D^{-1}$ in $(\cL(\widetilde X)^\Gamma_{K})^+$. 
	For each $s\in \mathbb R\backslash\{0\}$, the element
	\[ (\widetilde D^{-1} - si)(\widetilde D^{-1} + si) = \frac{1 + s^2 \widetilde D^2}{\widetilde D^{2}} \]
	is invertible in $(\cL(\widetilde X)^\Gamma_{K})^+$, with its inverse given by 
	\[ \frac{\widetilde D^{2}}{1 + s^2 \widetilde D^2} = \frac{1}{s^2}\Big(1-\frac{1}{1 + s^2\widetilde D^2}\Big) \in (\cL(\widetilde X)^\Gamma_{K})^+. \]
	It follows that 
	$ (\widetilde D^{-1} - si) $
	is invertible in $(\cL(\widetilde X)^\Gamma_{K})^+$ for each $s\in \mathbb R\backslash\{0\}$. 
	Equivalently,  $si \notin \sigma(\widetilde D^{-1})$ for all $s\in \mathbb R\backslash\{0\}$. 
	
	Therefore, for any $s\in \mathbb R\backslash\{0\}$, the Cayley transform
	$$\frac{\widetilde D +si}{\widetilde D -si}
	=1+\frac{2si}{\widetilde D -si}$$
	is an invertible element in $(\cL( \widetilde X)^\Gamma_{K})^+$.
	Furthermore, we have  
	$$\lim_{s\to 0}\Big\|\frac{\widetilde D +si}{\widetilde D -si}-1\Big\|_{1,K}=0, $$
	Thus for any $T>0$, the map $s\mapsto v_s$ is a continuous map from $[0, T]$ to $ (\cL(\widetilde X)^\Gamma_{K})^+$.
	
	Set $f(x)=\frac{1}{x-i}$. We have 
	$$\hat{f}(\xi)=-ie^{-\xi }\chi_{(0, \infty)}(\xi),$$
	where $\chi_{(0,\infty)}$ is the characteristic function of the interval $(0, \infty)$. 
	Note that $\hat f$  decays exponentially as $|\xi|\to\infty$. 
	Therefore by Lemma \ref{lemma:local_appro}, there exists some $s_0 >0$ such that the path  $s\mapsto f_{s}(\widetilde D)$ with $s\in[s_0,\infty)$ is an element of $\cL_L(\widetilde X)^\Gamma_{K}$. To summarize, we have proved that  the path
	$\{v_s\}_{s\in[0,\infty)}$
	is an invertible element in $(\cL_{L,0}(\widetilde X)^\Gamma_{K})^+$.
	
	Recall that the $l^1$-higher rho invariant  $\rho_{1,K}(\widetilde D)$ is defined to be $K$-theory class represented by $u = \{u_s\}_{s\in[0,\infty)}\in (\cB_{L,0}(\widetilde X)^\Gamma_{K})^+$ in line \eqref{eq:u_s}. We shall   show that 
	\begin{align*}
	[v]=[u]
	\in K_1(\cB_{L,0}(\widetilde X)^\Gamma_{K}).
	\end{align*}
	Note that for any fixed $s>0$, we have the following identity  in $(C^*(\widetilde X)^\Gamma)^+$:
	\begin{align*}
	v_s&=\exp\Big(\int_0^s
	\frac{2i\widetilde D}{\widetilde D^2+t^2}dt
	\Big)\\
	&=\exp\big(i\arctan(s^{-1}\widetilde D)-\pi i \sgn(s^{-1}\widetilde D)\big).
	\end{align*}
	
	Set 
	$$a_s=\int_0^s
	\frac{2i\widetilde D}{\widetilde D^2+t^2}dt=i\arctan(s^{-1}\widetilde D)-\pi i\sgn(s^{-1}\widetilde D).$$
	By the discussion above, we see that 
	\[ t \mapsto \frac{2i\widetilde D}{\widetilde D^2+t^2}\] 	
	is a continuous path from $[0, s]$ to $\cB(\widetilde X)^\Gamma_{K}$. It follows the integral 
	\[ \int_0^s	\frac{2i\widetilde D}{\widetilde D^2+t^2}dt \]
	absolutely	converges  in $\cB(\widetilde X)^\Gamma_{K}$. Therefore for any $T>0$, the map 
	\[ a\colon [0,T]\to \cB(\widetilde X)^\Gamma_{K} \textup{ by } a(s) = a_s\] is a continuous map. On the other hand, the function 
	\[ i\arctan(x)-\pi i\sgn(x) \] satisfies the assumption of Lemma \ref{lemma:local_appro}, which in particular implies  there exists $s_0'>0$ such that $\{a_{s}\}_{s\in[s_0',\infty)}\in \cB_L(\widetilde X)^\Gamma_{K}$. Therefore, the path $\{a_s\}_{s\in[0,\infty)}$ is an element of $\cB_{L,0}(\widetilde X)^\Gamma_{K}$.
	
	Recall that 
	\[ u_s
	=\exp(2\pi iF(s^{-1}\widetilde D_X)-2\pi iH(s^{-1}\widetilde D_X)), \]
	cf. the proof of Lemma $\ref{lemma:oddhighrho}$. 	Now the following homotopy of invertible elements 
	\begin{equation}	\label{eq:path}
	\lambda\mapsto 
	\exp\Big(\lambda a_s+(1-\lambda)\big(2\pi iF(s^{-1}\widetilde D)-2\pi iH(\widetilde D)\big) \Big),~\lambda\in[0,1]
	\end{equation}
	connects $v = \{v_{s}\}_{s\in[0,\infty)}$ and 
	$u = \{u_{s}\}_{s\in[0,\infty)}$ in $(\cB_{L,0}(\widetilde X)^\Gamma_{K})^+$. This finishes the proof. 
\end{proof}

To prove the product formula for $l^1$-higher rho invariants, we shall give an alternative but equivalent description of $l^1$-higher rho invariants in terms of Clifford-linear  Dirac operators. Recall the $\cl_{n}$-Dirac bundle  $\mathfrak S(X)$ over $X$: 
\begin{equation}\label{eq:cliffbdle}
	\mathfrak S(X) = P_{\spin}(X) \times_{\lambda} \cl_{n}
\end{equation} 
where $\lambda\colon \spin_n \to \End(\cl_n)$ is the representation given by left multiplication. Equip $\mathfrak S(X)$ with the canonical Riemannian connection  determined by the presentation $\lambda \colon P_\spin(X) \to \End(\cl_n)$. Lift all geometric data to $\widetilde X$ and denote by $\mathfrak S(\widetilde X)$ the associated the $\cl_{n}$-Dirac bundle on $\widetilde X$.  The decomposition $\cl_n = \cl_n^0 \oplus \cl_n^1$ gives rise to a parallel decomposition 
\[ \mathfrak S(\widetilde X) = \mathfrak S(\widetilde X)^0 \oplus \mathfrak S(\widetilde X)^1.  \]
Consider the (Clifford) volume element
$$\omega=i^{\frac{\dim X+1}{2}}e_1e_2\cdots e_n,$$
where $\{e_1,\cdots,e_n\}$ is an oriented local orthonormal frame of the tangent space. Let us denote by $E$ the Clifford multiplication by $\omega$.

 The Clifford-linear Dirac operator $\slashed D$ on $\widetilde X$ is defined to be the $\mathbb Z_2$-graded Dirac operator acting on $\mathfrak S(\widetilde X)$.  In fact, when $\dim X$ is odd,   it is equivalent to consider the following (ungraded) Dirac-type operator 
 \[ \slashed D E\colon C_c^\infty(\widetilde X, \mathfrak S(\widetilde X)^0) \to C_c^\infty(\widetilde X, \mathfrak S(\widetilde X)^0) \]
Note that $\slashed DE$ is self-adjoint in this case, since $\slashed D$ and $E$ commute when $\dim X$ is odd. With the natural identification $\cl_{n}^0 \cong \cl_{n-1}$,  the operator $\slashed DE$ becomes  $\cl_{n-1}$-linear Dirac operator. From now on, if $\dim X$ is odd, we shall write $\slashed D$ to denote the operator
\[ \slashed D \colon C_c^\infty(\widetilde X, \mathfrak S(\widetilde X)^0) \to C_c^\infty(\widetilde X, \mathfrak S(\widetilde X)^1),  \]
unless otherwise specified.

The $l^1$-higher rho invariant of the spin Dirac operator $\widetilde D$ on $\widetilde X$ coincides with the $l^1$-higher rho invariant of the Clifford-linear Dirac operator $\slashed D$. If $\dim X$ is odd,  it follows from Lemma \ref{lm:con-est} that the $l^1$-higher rho invariant of $\slashed D$ can be represented by 
$$\big[\{ (\slashed DE+si )(\slashed DE-s i)^{-1}\}_{s\in [0, \infty)}\big] =  \big[\{ (\slashed D+si E)(\slashed D-s iE)^{-1}\}_{s\in [0, \infty)}\big]$$
in $K_1(\cB_{L,0}(\widetilde X)^\Gamma_K).$ An alternative description of the $l^1$-higher rho invariant for the eve dimensional case is given in Lemma \ref{lm:projrepn} below. 



\begin{definition}
	Let $\cD(\widetilde X)^\Gamma_K$ be the $\|\cdot\|_{1,K}$-completion of operators in $\cB(H)^\Gamma$ with finite propagation and \emph{pseudo-local}, i.e., $\chi_A \circ T-T\circ \chi_A$ is compact for any precompact Borel set $A$. Similarly, we also  define  $\cD_L(\widetilde X)^\Gamma_K$ and $\cD_{L,0}(\widetilde X)^\Gamma_K$, which are the obvious analogues of  $\cB_L(\widetilde X)^\Gamma_K$ and $\cB_{L,0}(\widetilde X)^\Gamma_K$. 
\end{definition}
Obviously, $\cB(\widetilde X)^\Gamma_K$ (resp. $\cB_{L,0}(\widetilde X)^\Gamma_K$) is a closed two-sided ideal of $\cD(\widetilde X)^\Gamma_K$ (resp. $\cD_{L,0}(\widetilde X)^\Gamma_K$).
\begin{lemma}\label{lm:projrepn}
	For a given $K\geqslant 0$, suppose $\dim X$ is even and the scalar curvature $\kappa$ on $X$ satisfies that
	$$\inf_{x\in X}\kappa(x)>\frac{16(K_\Gamma+K)^2}{\tau^2}.$$ Let us denote 
	\[ P_{s, +} \coloneqq \varphi (\slashed D+sE) \textup{ and } P_{s, -} \coloneqq \varphi(\slashed D -sE)\]  
	where
	$$\varphi(x)=\begin{cases}
	1,&x>0,\\0,&x\leqslant 0.
	\end{cases}$$
	Then the elements $\{P_{s, +}\}_{s\in [0, \infty)}$ and $\{P_{s, -}\}_{s\in [0, \infty)}$
	lie in 
	$(\cD_L(\widetilde X)^\Gamma_K)^+$, and the element $\{P_{s, +} - P_{s, -}\}_{s\in [0, \infty)}$
	lies in $\cB_{L,0}(\widetilde X)^\Gamma_K$, Furthermore, we have
	\[\rho_{1,K}(g) = [P_{s, +}]-[P_{s, -}] \] 
	in $K_0(\cB_{L,0}(\widetilde X)^\Gamma_K)$.
\end{lemma}
\begin{proof}
	Note that $P_{s, +} = \varphi(\slashed D+sE)=\varphi(s^{-1}\slashed D+E)$ for any $s>0$. Recall the function $G$ in line \eqref{eq:G(x)} given by 
	$$G(x)=\frac{2}{\sqrt\pi}\int_{-\infty}^x e^{-y^2}dy-1.$$
	By Lemma \ref{lemma:G_tfinite}, Lemma \ref{lemma:G-sgn} and Lemma \ref{lm:con-est}, we have for any $r\geqslant 1$, the element
	$\{ G(r(s^{-1}\slashed D+E))\}_{s\in [0, \infty)}$ lies in $\cD_L(\widetilde X)^\Gamma_K,$
	and
	\begin{align*}
	&\|G(r(s^{-1}\slashed D+E))-\sgn(s^{-1}\slashed D+E)\|_{1,K}\\
	\leqslant& \int_r^\infty \|(s^{-1}\slashed D+E)e^{-s^{-1}t^2\slashed D^2-t^2}\|dt\\
	=&\int_r^\infty \frac{1}{t}C_1e^{-\frac{t^2\varepsilon}{s}}e^{-t^2}dt
	\leqslant  C_2\int_r^\infty e^{-t^2}dt,
	\end{align*}
	where $\sgn$ is the sign function and $C_1,C_2,\varepsilon$ are positive constants independent of $r$ and $s$. By varying $r$, we see that the element $\{\sgn(s^{-1}\slashed D+E)\}_{s\in [0, \infty)}$ lies in 
	$(\cD_L(\widetilde X)^\Gamma_K)^+.$ Since $P_{s, +} = \frac{\sgn(s^{-1}\slashed D+E) + 1}{2}$, it follows that $\{P_{s, +} \}_{s\in [0, \infty)}$ lies in 
	$(\cD_L(\widetilde X)^\Gamma_K)^+$. 
	Similarly, we also have $(s\mapsto P_+(\slashed D-sE)) \in \cD_L(\widetilde X)^\Gamma_K$. It follows from \cite[Lemma 5.8]{Higson1} that the difference
	$$\{P_{s, +} - P_{s, -}\}_{s\in [0, \infty)}$$
	lies in $\cB_{L,0}(\widetilde X)^\Gamma_K$.
	
	Now we shall adapt the proof of \cite[Theorem 5.5]{Higson2} to  show that 
	\[\rho_{1,K}(g) = [P_{s, +}]-[P_{s, -}]. \] 
	First, if we decompose $\mathfrak S(\widetilde X) = \mathfrak S(\widetilde X)^+ \oplus \mathfrak S(\widetilde X)^- $ with respect to the grading operator $E$, then the operator $E$ and $\slashed D$ become 
	\[  E = \begin{pmatrix}
	1 & 0 \\ 0 & -1 
	\end{pmatrix} \textup{ and } \slashed D = \begin{pmatrix}
	0 & \slashed D^- \\
	\slashed D^+ & 0 
	\end{pmatrix}. \] 
	For each $s>0$, let $\chi_s\colon \mathbb R \to [-1, 1]$ be the normalizing function given by 
	\[ \chi_s(x) = \frac{x}{\sqrt{x^2 + s^2}}.\]	
	With the decomposition $\mathfrak S(\widetilde X) = \mathfrak S(\widetilde X)^+ \oplus \mathfrak S(\widetilde X)^- $, we have  
	\[ \chi_s(\slashed D) = \begin{pmatrix}
	0 & V_s \\
	U_s & 0
	\end{pmatrix}   \]
	where $V_s = \slashed D^- (\slashed D^+\slashed D^- + s^2)^{-1/2}$ and $U_s = \slashed D^+ (\slashed D^-\slashed D^+ + s^2)^{-1/2}$. By the explicit construction of $\rho_{1,K}(g)$ (cf. formula \eqref{eq:idem}), we have 
	\begin{align*}
	\rho_{1,K}(g) &= \left[\begin{pmatrix}
	1-(1-U_sV_s)^2&(2-U_sV_s)U_s(1-V_sU_s)\\V_s(1-U_sV_s)&(1-V_sU_s)^2
	\end{pmatrix}\right] - \left[ \begin{pmatrix}
	1 & 0 \\ 0 & 0 
	\end{pmatrix}\right]\\
	& = \left[\begin{pmatrix}
	1-\left(\frac{s^2}{\slashed D^+\slashed D^- + s^2}\right)^2 & \frac{s^4\slashed D^+ }{(\slashed D^-\slashed D^+ + s^2)^{5/2}} + \frac{s^2\slashed D^+}{(\slashed D^-\slashed D^+ + s^2)^{3/2}}\\ 
	& \\
	\frac{s^2 \slashed D^- }{(\slashed D^+\slashed D^- + s^2)^{3/2}}  & \left(\frac{s^2}{\slashed D^-\slashed D^+ + s^2}\right)^2
	\end{pmatrix}\right] -  \left[ \begin{pmatrix}
	1 & 0 \\ 0 & 0 
	\end{pmatrix}\right]
	\end{align*}
	Note that the idempotent 
	\begin{align*}
	& 	\begin{pmatrix}
	1-(1-U_sV_s)^2&(2-U_sV_s)U_s(1-V_sU_s)\\V_s(1-U_sV_s)&(1-V_sU_s)^2
	\end{pmatrix}\\
	= &   \begin{pmatrix}
	1-\left(\frac{s^2}{\slashed D^+\slashed D^- + s^2}\right)^2 & \frac{s^4\slashed D^+ }{(\slashed D^-\slashed D^+ + s^2)^{5/2}} + \frac{s^2\slashed D^+}{(\slashed D^-\slashed D^+ + s^2)^{3/2}}\\ 
	& \\
	\frac{s^2 \slashed D^- }{(\slashed D^+\slashed D^- + s^2)^{3/2}}  & \left(\frac{s^2}{\slashed D^-\slashed D^+ + s^2}\right)^2
	\end{pmatrix} 
	\end{align*}
	is homotopic to the following projection
	\[\begin{pmatrix}
	1-\frac{s^2}{\slashed D^+\slashed D^- + s^2} & \frac{s\slashed D^+}{\slashed D^-\slashed D^+ + s^2}\\ 
	& \\
	\frac{s \slashed D^- }{\slashed D^+\slashed D^- + s^2}  & \frac{s^2}{\slashed D^-\slashed D^+ + s^2}
	\end{pmatrix}   \]
	through the path of idempotents
	\[  \begin{pmatrix}
	1-\left(\frac{s^2}{\slashed D^+\slashed D^- + s^2}\right)^{2- a} &  \frac{(\slashed D^-\slashed D^+ + s^2)^{2-a} - s^{2(2-a)}}{\slashed D^-\slashed D^+ (\slashed D^-\slashed D^+ + s^2)^{1-a}} \frac{s^{2-a}\slashed D^+}{(\slashed D^-\slashed D^+ + s^2)^{(3-a)/2}}\\ 
	& \\
	\frac{s^{2-a} \slashed D^- }{(\slashed D^+\slashed D^- + s^2)^{(3-a)/2}}  & \left(\frac{s^2}{\slashed D^-\slashed D^+ + s^2}\right)^{2-a}
	\end{pmatrix}.  \]
	Here we have implicitly used the fact that 
	\[ \left(\frac{s^2}{\slashed D^2 + s^2}\right)^b \textup{ lies in } (\cB_{L,0}(\widetilde X)^\Gamma_K)^+ \]
	for each $b> 0$.  In fact, similar to the proof of Lemma \ref{lm:con-est}, for any $b>0$, we have
	$$\frac{1}{(s^{-2}\slashed D^2+1)^b}=\frac{1}{C_b}\int_0^\infty
	e^{-t^{\frac 1 b}(s^{-2}\slashed D^2+1)}dt,
	$$
	where
	$$C_b=\int_0^\infty e^{-t^{\frac 1 b}}dt.$$
	By Lemma \ref{lemma:largespectralgap},  we have
	$$\|e^{-y\widetilde D^2}\|_{1,K}\leqslant Ce^{-y\varepsilon}$$
	for some $C>0$ and $\varepsilon>0$. Therefore, 
	\begin{equation}\label{eq:1/D2+1}
	\Big\|\frac{1}{(s^{-2}\slashed D^2+1)^b}\Big\|_{1,K}
	\leqslant \frac{1}{C_b}\int_0^\infty
	Ce^{-t^{\frac1b}(s^{-2}\varepsilon+1)}dt=
	\frac{C}{(s^{-2}\varepsilon +1)^b}\leqslant C.
	\end{equation}

	Set $S_1=E$ and $S_2=\{\sgn(\slashed D +sE)\}_{s\in [0, \infty)}$, which are elements in $\cD_{L}(\widetilde X)^\Gamma_K$ whose squares are equal to $1$. A straightforward calculation shows that  
	\[ \frac{-S_2S_1S_2+1}{2} = \begin{pmatrix}
	1-\frac{s^2}{\slashed D^+\slashed D^- + s^2} & \frac{s\slashed D^+}{\slashed D^-\slashed D^+ + s^2}\\ 
	& \\
	\frac{s \slashed D^- }{\slashed D^+\slashed D^- + s^2}  & \frac{s^2}{\slashed D^-\slashed D^+ + s^2}
	\end{pmatrix}.   \]
	From the above discussion, we conclude that 
	$$\rho_{1,K}(g)=\big[\frac{-S_2S_1S_2+1}{2}\big]-\big[\frac{S_1+1}{2}\big].$$
	On the other hand, using the fact that $\slashed D$ and $E$ anticommute, we have that 	$$[P_{s, +}]-[P_{s, -}]=\big[\frac{-S_1S_2S_1+1}{2}\big]- \big[\frac{S_2+1}{2}\big].$$
	
	Now set $U=S_1S_2$, which is obviously invertible. 
	\begin{claim*}
		The spectrum of $U$ excludes $\{-a:a\in\R,a\geqslant 0 \}$. 
	\end{claim*}
In fact, for any $a>0$, we have
	$$U+a=S_1S_2+aS_1^2=S_1(aS_1+S_2).$$
	It suffices to show that $(aS_1+S_2)^2$ is invertible. In fact, 
	\begin{align*}
	(S_2+aS_1)^2=&\Big(s\mapsto 1+a^2+\frac{2as}{\sqrt{D^2+s^2}}\Big)\\
	=&\Big(s\mapsto 2a\big(\frac{1+a^2}{2a}+\frac{1}{\sqrt{s^{-2}D^2+1}}\big)\Big)
	\end{align*}
	When $a>0$, we have $\frac{1+a^2}{2a}>1$. 
	
	Let $r$ be the spectral radius of $(s\mapsto \frac{1}{s^{-2}D^2+1})$. By the spectral radius theorem, we have
	$$r=\lim_{n\to\infty}\Big\|\frac{1}{(s^{-2}\slashed D^2+1)^\frac{n}{2}}\Big\|_{1,K}^{\frac 1 n}\leqslant \lim_{n\to\infty}\sqrt[n]{C}=1.$$
	Therefore, we see that $S_2+aS_1$ is invertible for any $a>0$. Our claim follows.
	
	Since the spectrum of $U$ excludes $\{-a:a\in\R,a\geqslant 0 \}$, $\sqrt{z}$ is a holomorphic function on the spectrum of $U$. Thus there is an invertible element $V\in \cD_L(\widetilde X)^\Gamma_K$ such that $V^2=U$.
	
	Note that 
	$$S_1U=U^{-1}S_1,~S_2U=U^{-1}S_2.$$
	As $S_1^2=S_2^2=1$, it follows from the holomorphic functional calculus that 
	$$S_1V=V^{-1}S_1,~S_2V=V^{-1}S_2.$$
	Set $W=VS_2$. Then we have
	$$W^{-1}S_1W=S_2V^{-1}S_1VS_2=S_2S_1V^2S_2=S_2$$
	and similarly $W^{-1}S_2W=S_1$. Therefore
	$$[P_{s, +}]-[P_{s, -}]= [W^{-1}\frac{-S_1S_2S_1+1}{2}W]-[W^{-1}\frac{S_2+1}{2}W]=\rho_{1,K}(g).$$ This finishes the proof.
\end{proof}

Now we are ready to prove Theorem \ref{thm:product}.
\begin{proof}[Proof of Theorem $\ref{thm:product}$]

Let us first prove the even dimensional case, that is, when $ {n = \dim X}$ is even. Recall that,  the (Clifford) volume element on $\widetilde X$ is given by 
$$\omega=i^{\frac{n}{2}}e_1e_2\cdots e_n,$$
where $\{e_1,\cdots,e_n\}$ is an oriented local orthonormal frame of the tangent space of $\widetilde X$. The corresponding volume element on $\widetilde X\times \mathbb R$ is given by 
$$\omega_{\widetilde X\times \mathbb R}=i^{\frac{n}{2} + 1}e_1e_2\cdots e_n e_{\mathbb R},$$
where $e_{\mathbb R}$ corresponds to the Clifford multiplication by the unit vector of $\mathbb R$. Let us denote by $E$ the operator determined by multiplication of $\omega$. 
    Recall the natural isomorphism: 
    \[ \cl_n \cong \cl_{n+1}^{0} \]
    by sending $e_j$ to $e_{n+1}e_j$ for all $1\leq j \leq n$. In particular,  this induces natural identifications 
     \[ \mathfrak S(\widetilde X) \cong \mathfrak S(\widetilde X\times \mathbb R)^0  \] 
     and 
     \[\mathfrak S(\widetilde X\times \mathbb R)^1 \xrightarrow{e_{\mathbb R}} \mathfrak S(\widetilde X\times \mathbb R)^0 \cong \mathfrak S(\widetilde X),  \]
    Under these  identifications, the operators
    \[ 1\hotimes D_\R+  \slashed D\hotimes 1 \pm  E\hotimes e_\R 
    \colon  C_c^\infty(\widetilde X \times \mathbb R, \mathfrak S(\widetilde X\times \mathbb R)^0) \to C_c^\infty(\widetilde X\times \mathbb R, \mathfrak S(\widetilde X\times \mathbb R)^1) \]
    become 
     \[ 1\otimes \frac{d}{dt}+  \slashed D\otimes 1 \pm  E\otimes 1 
     \colon  C_c^\infty(\widetilde X \times \mathbb R, \mathfrak S(\widetilde X)) \to C_c^\infty(\widetilde X\times \mathbb R, \mathfrak S(\widetilde X)). \]
    In particular,  the $l^1$-higher rho invariant $\rho_{1,K}(g+dt^2)$ is  represented by
	\begin{align*}
	&\Big[ (1\hotimes D_\R+  \slashed D\hotimes 1 + is E\hotimes i e_\R )(1\hotimes D_\R+  \slashed D\hotimes 1 - is E\hotimes i e_\R )^{-1}\Big]\\
	=&\Big[  (s^{-1}\otimes \frac{d}{dt}+  s^{-1}\slashed D\otimes 1 - E\otimes  1 )(s^{-1}\otimes \frac{d}{dt}+  s^{-1}\slashed D\otimes 1 +  E\otimes 1 )^{-1}\Big]\\
	=&\Big[ ( 1\otimes iD_s+  s^{-1}\slashed D\otimes 1 - E\otimes  1 )(1\otimes iD_s+  s^{-1}\slashed D\otimes 1 +  E\otimes 1 )^{-1}\Big]
	\end{align*}
	in $K_1(\cB_{L,0}(\widetilde X\times \R)^\Gamma_K)$, where we have 
	\[ D_s \coloneqq  (1+s)^{-1} \frac{1}{i} \frac{d}{dt}. \] 
	Note that the path 
	\begin{align*}
		&\frac{1\otimes iD_s+  s^{-1}\slashed D\otimes 1 - E\otimes  1 }{1\otimes iD_s+  s^{-1}\slashed D\otimes 1 +  E\otimes 1 }
		= 1 - \frac{2E\otimes 1}{1\otimes iD_s+  (s^{-1}\slashed D + E)\otimes 1 } 
	\end{align*}
is homotopic to the path 
\[ 1 - \frac{2E (s^{-2}\slashed D^2+1)^{-1/2}\otimes 1}{1\otimes iD_s+  (s^{-1}\slashed D + E) (s^{-2}\slashed D^2+1)^{-1/2}\otimes 1 }  \]
through the family of paths 
\begin{align*}
	&\frac{1\otimes iD_s+  (s^{-1}\slashed D - E) (s^{-2}\slashed D^2+1)^{-r/2}\otimes 1  }{1\otimes iD_s+  (s^{-1}\slashed D + E)(s^{-2}\slashed D^2+1)^{-r/2}\otimes 1  }\\
	=& 1 - \frac{2E (s^{-2}\slashed D^2+1)^{-r/2}\otimes 1}{1\otimes iD_s+  (s^{-1}\slashed D + E) (s^{-2}\slashed D^2+1)^{-r/2}\otimes 1 }
\end{align*}
where $s\in [0, \infty)$ and $r\in [0,1]$. Let us denote 
\[b_r(s) =  \begin{cases}
{\displaystyle \frac{2E (s^{-2}\slashed D^2+1)^{-r/2}\otimes 1}{1\otimes iD_s+  (s^{-1}\slashed D + E) (s^{-2}\slashed D^2+1)^{-r/2}\otimes 1 } } & \textup{ if }  s >0, \\ 
& \\
	0  & \textup{ if } s = 0.
\end{cases}   \] 
\begin{claim*}
For each $r\in [0, 1]$, the path 
\[ \big\{b_r(s)\big\}_{s\in [0, \infty)} \textup{ lies in } \cB_{L,0}(\widetilde X\times \R)^\Gamma_K.  \]
Moreover, the family $\{b_r\}_{r\in [0, 1]}$ is continuous in $r$. 
\end{claim*}
Recall that 
\[ (s^{-1}\slashed D \pm E)^2 = s^{-2}\slashed D^2 + 1. \]
Therefore, there exists a constant $\delta>0$ such that 
\begin{align*}
	& \frac{2E (s^{-2}\slashed D^2+1)^{-r/2}\otimes 1}{1\otimes iD_s+  (s^{-1}\slashed D + E) (s^{-2}\slashed D^2+1)^{-r/2}\otimes 1 }  \\
= 	&  \frac{\big(2E (s^{-2}\slashed D^2+1)^{-r/2}\otimes 1\big) \cdot  \big(-1\otimes iD_s+  (s^{-1}\slashed D + E) (s^{-2}\slashed D^2+1)^{-r/2}\otimes 1\big)}{1\otimes D_s^2+  (s^{-2}\slashed D^2+1)^{1-r}\otimes 1 } \\
= &  \frac{\big(2E (s^{-2}\slashed D^2+1)^{-r/2}\otimes 1\big) \cdot  \big(-1\otimes iD_s+  (s^{-1}\slashed D + E) (s^{-2}\slashed D^2+1)^{-r/2}\otimes 1\big)}{1\otimes (D_s^2 + \delta) +  ( (s^{-2}\slashed D^2+1)^{1-r} - \delta)\otimes 1 }  \\
= & \frac{- 2E (s^{-2}\slashed D^2+1)^{-r/2}\otimes iD_s (D_s^2 + \delta)^{-1} + 2E (s^{-1}\slashed D + E) (s^{-2}\slashed D^2+1)^{-r}\otimes  (D_s^2 + \delta)^{-1}}{1\otimes 1 +  ( (s^{-2}\slashed D^2+1)^{1-r} - \delta)\otimes (D_s^2 + \delta)^{-1}  }, 
\end{align*}
where the last term lies in $\cB_{L,0}(\widetilde X\times \R)^\Gamma_K$ by  estimates similar to those in the proof of Lemma \ref{lm:projrepn} (cf. line \eqref{eq:1/D2+1}). Furthermore, the same estimates also show that  the family $\{b_r\}_{r\in [0, 1]}$ is continuous in $r\in [0, 1]$.

Let $P_{s, +}$ be the positive projection of $s^{-1}\slashed D+E$, that is, $P_{s, +} = \varphi(s^{-1}\slashed D+E)$,  where  
\[\varphi(x)=\begin{cases}
	1, & \textup{ if } x\geqslant 0,\\ 
	0,& \textup{ if } x < 0.
\end{cases} \]
Let $Q_{s, +}=1 -P_{s, +}$. Similarly, we define $P_{s,-} = \varphi(s^{-1}\slashed D-E)$ and $Q_{s, -}=1-P_{s, -}$.
Then we have 
\begin{align*}
	b_1(s) & =  1 - \frac{2E (s^{-2}\slashed D^2+1)^{-1/2}\otimes 1}{1\otimes iD_s+  (s^{-1}\slashed D + E) (s^{-2}\slashed D^2+1)^{-1/2}\otimes 1 } \\
	 & = \frac{1\otimes iD_s+  (s^{-1}\slashed D - E) (s^{-2}\slashed D^2+1)^{-1/2}\otimes 1}{1\otimes iD_s+  (s^{-1}\slashed D + E) (s^{-2}\slashed D^2+1)^{-1/2}\otimes 1 } \\
	 & = \frac{1\otimes iD_s+  (P_{s, -} - Q_{s, -})\otimes 1}{1\otimes iD_s+  (P_{s, +} - Q_{s, +})\otimes 1 } \\
	 & = \frac{P_{s,-}\otimes (iD_s + 1) +  Q_{s, -}\otimes (iD_s - 1)}{P_{s,+}\otimes (iD_s + 1) +  Q_{s, +}\otimes (iD_s - 1) }
\end{align*}
A straightforward calculation shows that 
\begin{align*}
	& \big(P_{s,+}\otimes (iD_s + 1) +  Q_{s, +}\otimes (iD_s - 1)\big)^{-1} \\ 
	= &  P_{s,+}\otimes (iD_s + 1)^{-1} +  Q_{s, +}\otimes (iD_s - 1)^{-1}.
\end{align*}
It follows that 
\begin{align*}
& 	\frac{P_{s,-}\otimes (iD_s + 1) +  Q_{s, -}\otimes (iD_s - 1)}{P_{s,+}\otimes (iD_s + 1) +  Q_{s, +}\otimes (iD_s - 1) } \\
 = 	& P_{s, -}P_{s, +} \otimes 1 +  P_{s, -} Q_{s, +} \otimes (iD_s + 1)(iD_s - 1)^{-1}  \\
 & \quad + Q_{s, -}P_{s, +} \otimes (iD_s - 1)(iD_s + 1)^{-1} + Q_{s, -}Q_{s, +} \otimes 1 \\
 =  & \big(P_{s, -} \otimes  (iD_s + 1)(iD_s - 1)^{-1}  + (1- P_{s, -})\otimes 1 \big)  \\
  & \quad \cdot \big(P_{s, +} \otimes  (iD_s - 1)(iD_s + 1)^{-1}  + (1- P_{s, +})\otimes 1 \big)\\
\end{align*}
whose $K$-theory class is precisely the $K$-theory product 
\[  ([P_{s, +}] - [P_{s, -}]) \otimes [(D_s + i)(D_s - i)^{-1}] = \rho_{1,K}(g)\otimes \ind_L(D_\R). \]

Now let us prove the odd dimensional case, that is,  when $n = \dim X$ is odd.  Let $g+dx^2+ dy^2$ be the product metric on $X\times \mathbb R^2$ determined by the metric $g$ on $X$. 
\begin{claim*}
The $l^1$-higher rho invariant $\rho_{1, K}(g + dx^2 + dy^2)$ associated to the metric $g+dx^2+ dy^2$ satisfies the following product formula: 
\begin{equation}\label{eq:oddprod}
\rho_{1,K}(g+dx^2+dy^2)=\rho_{1,K}(g)\otimes \ind_L(D_{\R^2})
\end{equation}
in $K_1(\cB_{L,0}(\widetilde X\times \R^2)^\Gamma_K)$, where $\otimes$ stands for the $K$-theory product 
$$_-\otimes_-\colon K_1(\cB_{L,0}(\widetilde X)^\Gamma_K)\otimes K_0(C^*_L(\R^2))\to K_1(\cB_{L,0}(\widetilde X\times \R^2)^\Gamma_K).$$
\end{claim*} 
The proof of this claim is similar to the even dimensional case (cf. \cite[Proposition D.3]{Weinberger:2016dq}), and we shall omit the details. Now by the product formula from the even dimensional case, we have 
\[ \rho_{1,K}(g+dx^2+dy^2)=\rho_{1,K}(g + dx^2)\otimes \ind_L(D_{\R}).\] Recall that we also have the following product formula for the Dirac operators in the $C^\ast$-algebraic setting:
\[ \ind_L(D_{\R^2}) = \ind_L(D_{\R})\otimes \ind_L(D_{\R}).\]
These, together with the formula \eqref{eq:oddprod}, imply that 
\[ \rho_{1,K}(g+dx^2)=\rho_{1,K}(g)\otimes \ind_L(D_{\R}) \]
for the case where $\dim X$ is odd. This finishes the proof. 
\end{proof}

\section{Higher index on manifolds with boundary}\label{sec:1}
In this section, we review the construction of higher indices of Dirac operators on spin manifolds with boundary whose metrics are assumed to have product structure near the boundary and also have positive scalar curvature near the boundary. This definition was originally introduced in \cite{Roecoarse,Bunke} and is more detailed stated in \cite{Xiepos}. 

\begin{geoset}\label{geoset}
	Let $M$ be a compact spin manifold with boundary $\partial M$. Assume that the metric $g$ on $M$ has product structure near the boundary and we denote by $g_\partial$ the restriction of $g$ to $\partial M$. Let $M_\infty$ be the  manifold obtained by attaching an infinite cylinder to $M$, that is,  $M_\infty=M\cup_{\partial M}\partial M\times[0,\infty)$. Since $g$ has product structure near the boundary, $g$ extends canonically to a Riemannian metric $g_\infty$ on $M_\infty$, where  $g_\infty=g_\partial+dt^2$  on the cylindrical part. 
	
	Now let $\widetilde M_\infty$ be the universal cover of $M_\infty$ and $\Gamma$ the fundamental group of $M_\infty$, which is also the fundamental group of $M$. If we denote the universal cover of $M$ by $\widetilde M$, then we have $\widetilde M_\infty = \widetilde M\cup_{\partial \widetilde M}\partial \widetilde M\times [0,\infty)$.
	Let $S(\widetilde M_\infty)$ be the spinor bundle on $\widetilde M_\infty$. Endow $\widetilde M_\infty$ with the metric $\widetilde g_\infty$ which is the lift of the metric $g_\infty$ on $M_\infty$. The fundamental group $\Gamma$  acts isometrically on $\widetilde M_\infty$ and $S(\widetilde M_\infty)$.  Let $\widetilde D$ be the associated Dirac operator on $\widetilde M_\infty$. Note that  $\widetilde D$ is $\Gamma$-equivariant, that is,
	$$\widetilde D(\gamma f)=\gamma \widetilde Df, \forall\gamma\in\Gamma, $$
	where $f$ is any compactly support smooth section of $S(\widetilde M_\infty)$.
	
	For $0\leqslant s<\infty$, we denote $M_s=M\cup_{\partial M}\partial M\times[0,s]\subset M_\infty$ and 
	$\widetilde M_s = \widetilde M\cup_{\partial \widetilde M}\partial \widetilde M\times [0,s]\subset \widetilde M_\infty.$
%
\end{geoset}

Using the Hilbert space $L^2(\widetilde M_\infty,S(\widetilde M_\infty))$, the Hilbert space of all square-integrable sections of $S(\widetilde M_\infty)$, we define the $C^*$-algebra $C^*(\widetilde M_\infty)^\Gamma$ as in Definition \ref{def:Roe}.
\begin{definition}\label{def:supporton}
	 We say an operator  $T\in B(H)^\Gamma$ is supported on $\widetilde M_{s}$ if $$\chi_{\partial \widetilde M\times[s,+\infty)} \circ T\text{ and }T\circ \chi_{\partial \widetilde M\times[s,+\infty)}$$
	  are both zero. We define $C^*(\widetilde M_s)^\Gamma$ to be the operators in $C^*(\widetilde M_\infty)^\Gamma$ supported on $\widetilde M_s$. 
\end{definition}
For any $0<s<\infty$, we have 
$$C^*(\widetilde M_s)^\Gamma\cong \cK\otimes C^*_r(\Gamma).$$
\subsection{Even dimensional case}\label{subsec:evenindex}
Let us assume the same notation as in the Geometric Setup $\ref{geoset}$.  Recall that when $\dim M$ is even, the spinor bundle on $M_\infty$ admits a natural $\mathbb Z_2$-grading and the Dirac operator $D$ on $M_\infty$ is an odd operator with respect to this $\mathbb Z_2$-grading.

Recall the definition of normalizing function given in Definition \ref{def:normalizingfunc}. Given a normalizing function $F$, we obtain $p_{F(\widetilde D)}$ as in line \eqref{eq:idem}.

The difference from the definition of the higher index for closed manifold in Section \ref{sec:higherindexclosed} is the following cut-off function.
\begin{definition}\label{def:cutoff}
Let $\Psi$ be a smooth cut-off function on $\widetilde M_\infty$ such that $|\Psi(x)|\leqslant 1$, $\Psi(x)=1$ for all $x\in\widetilde M_0 = \widetilde M$ and $\Psi(x)=0$ on $\partial\widetilde M\times[1,+\infty)$. Also, we choose $\Psi$ so that  on the subspace  $\partial\widetilde M\times [0,1]$, it only depends on the parameter $t\in[0,1]$. We furthermore define $\Psi_T$ on $\widetilde M_\infty$ such that $\Psi_T=1$ on $\widetilde M_T$ and $\Psi_T(x,t)=\Psi(x,t-T)$ for $(x, t)\in\partial\widetilde M\times [T,+\infty)$.
\end{definition}

We define
\begin{equation}\label{eq:qFDT}
q_{F(\widetilde D),T}=\big(p_{F(\widetilde D)}-\begin{psmallmatrix}
1&0\\0&0
\end{psmallmatrix}\big)\Psi_T + \begin{psmallmatrix}
1&0\\0&0
\end{psmallmatrix}.
\end{equation}
Clearly, $q_{F(\widetilde D),T}$ lies in $M_2((C^*(\widetilde M_{T+5N_F})^\Gamma)^+)$, where $(C^*(\widetilde M_{T+5N_F})^\Gamma)^+$ is the unitalization of $C^*(\widetilde M_{T+5N_F})^\Gamma$.
\begin{lemma}\label{lemma:q^2-q}
	With the same setup as above, let $\widetilde D$ be the Dirac operator on $\widetilde M_\infty$. Then there exist  $T >0$ and  a smooth normalizing function $F$ whose Fourier transform is compactly supported such that 
	\begin{equation}\label{eq:quasi-idem}
\|q_{F(\widetilde D),T}^2-q_{F(\widetilde D),T}\|_\op<1/2.
	\end{equation}
\end{lemma}
\begin{proof}
	Let  $\varphi$ be a smooth normalizing function whose Fourier transform is compactly supported. For each $t >0$, we define a normalizing  function  $\varphi_t$  by setting $\varphi_t(x)=\varphi(tx)$ for all $x\in \mathbb R$. Let $\widetilde D_{\cyc}$ be the Dirac operator on $\partial\widetilde M\times\R$, where the metric $g_{\partial\widetilde M\times \mathbb R}$ is given by $g_\partial+dt^2$. By the Lichnerowicz formula 
	\[ \widetilde D_{\cyc}^2 = \nabla^\ast \nabla + \frac{\kappa}{4} \]  where $\kappa$ is scalar curvature of the metric $g_{\partial\widetilde M\times \mathbb R}$,  we see that  $\widetilde D_{\cyc}$ is an invertible operator, since the metric $g_\partial$ has positive scalar curvature,
	
	It follows from the Fourier inverse transform formula (cf. line $\eqref{eq:fourier}$) that $\varphi_t(\widetilde D)$ has propagation $\leqslant N_\varphi \cdot t$. Moreover, the operator
	$$p_{\varphi_t(\widetilde D)}-q_{\varphi_t(\widetilde D),5N_\varphi t}=\big(p_{\varphi_t(\widetilde D)}-\begin{psmallmatrix}
	1&0\\0&0
	\end{psmallmatrix}\big)(1-\Psi_{5tN_\varphi })$$
	 on $\widetilde M_\infty$ coincides with the operator 
	 	$$\big(p_{\varphi_t(\widetilde D_{\cyc})}-\begin{psmallmatrix}
	 1&0\\0&0
	 \end{psmallmatrix}\big)(1-\Psi_{5tN_\varphi })$$
	 on $\partial \widetilde M\times\mathbb R$ under the canonical identification between the subspace $\partial\widetilde M\times\R_{\geqslant 0} $ of  $\partial\widetilde M\times\R$ and the cylindrical part $\partial\widetilde M\times\R_{\geqslant 0} $ of $\widetilde M_\infty$. Since $\widetilde D_{\cyc}$ is invertible, we have
	 \begin{align*}
	 &\lim_{t\to\infty}\max\{\|1-\varphi_t(\widetilde D_{\cyc})^-\varphi_t(\widetilde D_{\cyc})^+\|_\op,\|1-\varphi_t(\widetilde D_{\cyc})^+\varphi_t(\widetilde D_{\cyc})^-\|_\op \}\\
	 =&\lim_{t\to\infty}\|\varphi_t(\widetilde D_{\cyc})^2-1\|_\op=0.
	 \end{align*}
	Since $\|\varphi_t(\widetilde D_{\cyc})^\pm\|_\op\leqslant 1$ for all $t>0$, we have
	\begin{equation}
	\lim_{t\to\infty}\|p_{\varphi_t(\widetilde D)}-q_{\varphi_t(\widetilde D),5tN_\varphi}\|_\op=0.
	\end{equation}
	Moreover, we have $\|p_{\varphi_t(\widetilde D)}\|_\op\leqslant 64$  for all $t>0$ and $p_{\varphi_t(\widetilde D)}^2 = p_{\varphi_t(\widetilde D)}$. It follows that 
	\begin{equation}
	\lim_{t\to\infty}\|q_{\varphi_t(\widetilde D),5N_\varphi t}^2-q_{\varphi_t(\widetilde D),5tN_\varphi}\|_\op=0.
	\end{equation}
	Now the proof is finished by setting $F = \varphi_t$ and $T = 5tN_\varphi$ for some sufficiently large $t\gg  0$. 
\end{proof}

Let $\Theta$ be the function on the complex plane defined by
$$\Theta(z)=\begin{cases}
0& \mathrm{Re}z\leqslant 1/2,\\
1& \mathrm{Re}z>1/2.
\end{cases}$$
Let $q_{F(\widetilde D),T}$ be the operator constructed in the above lemma. It follows that $\Theta$ is holomorphic on the spectrum of $q_{F(\widetilde D),T}$. Therefore $\Theta(q_{F(\widetilde D),T})$ is an idempotent in $M_2((C^*(\widetilde M_{T+5N_F})^\Gamma)^+)$ and $\Theta(q_{F(\widetilde D),T})-\begin{psmallmatrix}
1&0\\0&0
\end{psmallmatrix}$ lies in $M_2(C^*(\widetilde M_{T+5N_F})^\Gamma)$.
\begin{definition}\label{def:evenindex}
	The higher index $\ind^\Gamma(D)$ of $\widetilde D$ on $\widetilde M_\infty$ is defined to be 
	\begin{equation}
		[\Theta(q_{F(\widetilde D),T})]-[\begin{psmallmatrix}1&0\\0&0\end{psmallmatrix}] \in K_0(C^*(\widetilde M_{T+5N_F})^\Gamma) \cong  K_0(C_r^*(\Gamma)).
	\end{equation}
\end{definition}

The higher index $\ind^\Gamma(D)$ is independent of the choice of $T$ and $F$, as long as the inequality $\eqref{eq:quasi-idem}$ in Lemma \ref{lemma:q^2-q} is satisfied.

\subsection{Odd dimensional case}\label{subsec:oddindex}
Let $\dim M$ be odd and $\widetilde D$ the associated Dirac operator on $\widetilde M_\infty$. 

Consider the function $P=(F+1)/2$ where  $F$ is a normalizing function. More precisely, we assume $P$ is a continuous function on $\R$ satisfying
\begin{enumerate}
	[label=(P\arabic*)]
	\item $\lim_{x\to-\infty}P(x)=0$ and $\lim_{x\to+\infty} P(x)=1;$
	\item and the Fourier transform of $P'$ is supported on the interval  $[-N_P,N_P]$ for some $N_P>0$.
\end{enumerate}
Again, it follows  the Fourier inverse transform formula (cf. line $\eqref{eq:fourier}$) that $P(\widetilde D)$ has propagation no more than $N_P$.

We define 
\begin{equation}\label{eq:f_n}
\begin{split}
f_n(x)=&\sum_{k=1}^n\frac{(2\pi i )^k}{k!}x^k-\Big(\sum_{k=1}^n\frac{(2\pi i)^k}{k!}\Big)x\\
=&\Big(\sum_{k=1}^n\frac{(2\pi i)^k}{k!}\Big)(x^2-x)+\sum_{k=1}^n\frac{(2\pi i)^k}{k!}\sum_{j=0}^{k-2}x^j(x^2-x),\\
\end{split}
\end{equation}
which satisfies
\begin{equation}\label{eq:e^{}-fn}
e^{2\pi ix}-(1+f_n(x))=\sum_{k=n+1}^\infty \frac{(2\pi i)^k}{k!}x^k+\Big(\sum_{k=1}^n\frac{(2\pi i)^k}{k!}\Big)x.
\end{equation}
Note that $f_n(P(\widetilde D))$ is locally compact and has propagation $\leqslant n\cdot N_P$.

\begin{lemma}\label{lemma:fninvertible}
	With the same notation as above, there exist $n >0$, $T >0$ and a continuous function $P$ satisfying conditions \textup{(P1)-(P2)} above such that 
	\begin{equation}\label{eq:inv}
\|e^{2\pi iP(\widetilde D)}-\big(1+f_n(P(\widetilde D))\Psi_T\big)\|_\op<1.
	\end{equation}
	where $\Psi_T$ is a compactly supported smooth function on $\widetilde M_\infty$ as Definition \ref{def:cutoff}.  
\end{lemma}
\begin{proof}
Let $\varphi$ be a smooth function on $\mathbb R$ satisfying conditions \textup{(P1)-(P2)}. Define the function $\varphi_t$ by setting $\varphi_t(x)=\varphi(tx)$. 

By the Fourier inverse transform formula (cf. line $\eqref{eq:fourier}$), $\varphi_t(\widetilde D)$ has propagation $\leqslant tN_{\varphi}$. Hence $f_n(\varphi_t(\widetilde D))$ has propagation $\leqslant n tN_{\varphi}$. Moreover, the operator  $f_n(\varphi_t(\widetilde D))(1-\Psi_{ntN_{\varphi}})$ on $\widetilde M_\infty$ coincides with the operator  $f_n(\varphi_t(\widetilde D_{\cyc}))(1-\Psi_{ntN_{\varphi}})$ on $\partial \widetilde M\times\mathbb R$ under the canonical identification between the subspace $\partial\widetilde M\times\R_{\geqslant 0} $ of  $\partial\widetilde M\times\R$ and the cylindrical part $\partial\widetilde M\times\R_{\geqslant 0} $ of $\widetilde M_\infty$.
	
	Since $\|\varphi_t(\widetilde D)\|_\op\leqslant 1$ and $\|\varphi_t(\widetilde D_{\cyc})\|_\op\leqslant 1$, we have for any $t>0$
	$$\|e^{2\pi i\varphi_t(\widetilde D)}-\big(1+f_n(\varphi_t(\widetilde D)) 
	\big)\|_\op<\frac 1 3
	\text{ and }\|e^{2\pi i\varphi_t(\widetilde D_{\cyc})}-\big(1+f_n(\varphi_t(\widetilde D_{\cyc}))\big)\|_\op<\frac 1 3
	$$
	when $n$ is sufficiently large.
	
	As $\widetilde D_{\cyc}$ is invertible, we also have 
	$$
	\|e^{2\pi \varphi_t(\widetilde D_{\cyc})}-1\|_\op<\frac{1}{3}.
	$$
	as long as $t$ is sufficiently large. 
	Therefore, when both $n$ and $t$ are sufficiently large, we have 
	\begin{align*}
	&\|e^{2\pi i \varphi_t(\widetilde D)}-\big(1+f_n(\varphi_t(\widetilde D))\Psi_{n tN_\varphi}\big)\|_\op \\
	<&	\|e^{2\pi i \varphi_t(\widetilde D)}-\big(1+f_n(\varphi_t(\widetilde D))\big)\|_\op+\|f_n(\varphi_t(\widetilde D))(1-\Psi_{ntN_\varphi  })\|_\op\\
	=&	\|e^{2\pi i \varphi_t(\widetilde D)}- 1 - f_n(\varphi_t(\widetilde D))\|_\op+\|f_n(\varphi_t(\widetilde D_{\cyc}))(1-\Psi_{nt N_\varphi})\|_\op\\
	<&	\|e^{2\pi i \varphi_t(\widetilde D)}-1 - f_n(\varphi_t(\widetilde D))\|_\op\\
	& \quad  +\|1 + f_n(\varphi_t(\widetilde D_{\cyc}))-e^{2\pi i\varphi_t(\widetilde D_{\cyc})}\|_\op  +\|e^{2\pi i \varphi_t(\widetilde D_{\cyc})}-1\|_\op<1.
	\end{align*}	
Now the proof is finished by setting $P = \varphi_t$ and $T = n N_\varphi \cdot t$ for some sufficiently large $n$ and $t$. 
\end{proof}

Now let $(P,n,T)$ be the triple chosen as in Lemma \ref{lemma:fninvertible} above so that the inequality $\eqref{eq:inv}$ is satisfied. Note that $\|e^{2\pi iP(\widetilde D)}\|_\op = 1$.  It follows that $1+f_n(P(\widetilde D))\Psi_T$  is invertible in $(C^*(\widetilde M_{\infty})^\Gamma)^+$. Moreover, $1+f_n(P(\widetilde D))\Psi_T$ lies in $(C^*(\widetilde M_{T+nN_P})^\Gamma)^+$ by construction, hence is also invertible in $(C^*(\widetilde M_{T+nN_P})^\Gamma)^+$, since $(C^*(\widetilde M_{T+nN_P})^\Gamma)^+$ is a $C^\ast$-subalgebra of $(C^*(\widetilde M_{\infty})^\Gamma)^+$.

\begin{definition}\label{def:oddindex}
	The higher index $\ind^\Gamma(D)$ of $\widetilde D$ is defined to be 
	\[ [1+f_n(P(\widetilde D))\Psi_T] \in 
 K_1(C^*(\widetilde M_{T+nN_P})^\Gamma) \cong K_1(C_r^*(\Gamma)). \]
\end{definition}

The higher index $\ind^\Gamma(D)$ is independent of the choice of $(P,n,T)$, as long as the inequality $\eqref{eq:inv}$ in  Lemma \ref{lemma:fninvertible} is satisfied.

\section{Sufficiently large spectral gap on the boundary and $l^1$-index of Dirac operator}\label{sec:2}
In this section, we construct the $l^1$-indices of Dirac operators  on spin manifolds whose scalar curvatures are sufficiently large on their boundaries.  As a consequence, we prove Theorem $\ref{thm:main}$, which is in fact a slightly improved version of Theorem \ref{theorem:main}.

Let us assume the same notation as in the Geometric Setup $\ref{geoset}$. Let $\mathcal F_\partial$ be a fundamental domain of the $\Gamma$-action on $\partial\widetilde M$. Fix a finite symmetric generating set $S$ of $\Gamma$. Let $\ell$ be the length function on $\Gamma$ induced by $S$. We define 
\begin{equation}
	\tau_\partial=\liminf_{\ell(\gamma)\to\infty}\sup_{x\in\mathcal F_\partial}\frac{\dist(x,\gamma x)}{\ell(\gamma)}
\end{equation}
and
\begin{equation}\label{eq:taupartial}
	K_\Gamma=\inf\{K:\exists C>0 \text{ s.t. }\#\{\gamma\in\Gamma:\ell(\gamma)\leqslant n\}\leqslant Ce^{Kn} \}.
\end{equation} 

\begin{theorem}\label{thm:main}
With the same notation as in the Geometric Setup $\ref{geoset}$, if there exists a constant $K\geqslant 0$ such that the scalar curvature $\kappa$ of the metric $g_\partial$ on $\partial M$ satisfies 
\begin{equation}\label{eq:scbd}
\inf_{x\in\partial M}\kappa(x)>\frac{16(K_\Gamma+K)^2}{\tau^2_\partial},
\end{equation}
then the higher index $\ind^\Gamma(D)$ of the Dirac operator $\widetilde D$  on $\widetilde M_\infty$ admits a natural preimage $\ind^\Gamma_{1,K}(D)$ in $K_n(l^1_K(\Gamma))$ under the inclusion 
\[ l^1_K(\Gamma) \hookrightarrow C^*_r(\Gamma). \] 
 where the Banach algebra $l^1_K(\Gamma)$ is defined in  Definition $\ref{def:konenorm}$. 
\end{theorem}

By the Lichnerowicz formula 
\[ \widetilde D_{\cyc}^2 = \nabla^\ast \nabla + \frac{\kappa}{4}, \] 
the lower bound on $\kappa$ in line $\eqref{eq:scbd}$ implies that
$$|\widetilde D_{\cyc}|>\frac{2(K_\Gamma+K)}{\tau_\partial},$$
where $\widetilde D_{\cyc}$ is the Dirac operator on $\partial \widetilde M\times \mathbb R$ equipped with the metric $\widetilde g_\partial + dt^2$. 

As in Definition \ref{def:Roe} and Definition \ref{def:konenorm}, we define the Banach algebras $\cB(\widetilde M_\infty)^\Gamma_K$.
Although the $\Gamma$-action on $\widetilde M_\infty$ is not cocompact, the definitions make sense due to the following observations.
\begin{enumerate}
	\item To define the $l^1$-norm on $\C(\widetilde X\times\R)^\Gamma$, we should use a fundamental domain $\mathcal F_\infty\subset \widetilde M_\infty$ such that
	\begin{itemize}
		\item $\mathcal F_\infty\cap \widetilde M$ is a precompact fundamental domain of the $\Gamma$-action on $\widetilde M$;
		\item $\mathcal F_\infty\cap(\partial\widetilde M\times[0,+\infty))=\mathcal F_\partial\times[0,+\infty)\subset \widetilde M_\infty,$
		where $\mathcal F_\partial$ is a precompact fundamental domain of the $\Gamma$-action on $\partial\widetilde M$.
	\end{itemize}
	\item Given a length function on $\Gamma$, we still have 
	\begin{equation}\label{eq:quasi}
	\dist(x,\gamma x)>\tau'\ell(\gamma)-C_0',~\forall\gamma\in \Gamma,~\forall x\in \mathcal F_\infty
	\end{equation}
	for some positive constants $C_0$ and $\tau$.
\end{enumerate}
Similarly to Definition \ref{def:supporton}, we define $\cB(\widetilde M_s)^\Gamma_K$ to be the operators in $\cB(\widetilde M_\infty)^\Gamma_K$ supported on $\widetilde M_s$. For any $0<s<\infty$, we have 
$$\cB(\widetilde M_s)^\Gamma_K\cong \cK\otimes l^1_K(\Gamma).$$

\subsection{Even dimensional case}\label{subsec:evenl1}

Let us prove  Theorem \ref{thm:main} in the case where $\dim M$ is even.  

We assume the same notation as in the Geometric setup $\ref{geoset}$. In particular, $\widetilde D$ is the Dirac operator on $\widetilde M_\infty$ and $\widetilde D_{\cyc}$ is the Dirac operator on the cylinder $\partial\widetilde M\times \R$.

	Let $q_{F(\widetilde D),T}$ be the element defined in line \eqref{eq:qFDT}, where $T$ is some positive number and $F$ is a smooth normalizing function with $F'$ a Schwartz function. In particular, $\|q_{F(\widetilde D),T}\|_{1, K}$ is finite, since the Fourier transform of $F$ has compact support.

	Similar to the construction of the higher index $\ind^\Gamma(D)\in K_0(C_r^\ast(\Gamma))$ as in Definition $\ref{def:evenindex}$. The proof of Theorem $\ref{thm:main}$ for the even dimensional case can be reduced to the following claim, which is an $l^1$-analogue of Lemma \ref{lemma:q^2-q}. 
	\begin{claim}\label{prop:q^2-q 1-norm}
		If $|\widetilde D_{\cyc}|>\frac{2(K_\Gamma+K)}{\tau_\partial}$, then there exist  a positive number $T$ and  a smooth normalizing function $F$  whose Fourier transform is compactly supported such that  
		$$\|q_{F(\widetilde D),T}^2-q_{F(\widetilde D),T}\|_{1, K}<1/2,$$
			where $\|\cdot\|_{1, K}$ is the weighted $l^1$-norm given in Definition $\ref{def:konenorm}$.
	\end{claim} 
\begin{proof}[Proof of Claim $\ref{prop:q^2-q 1-norm}$]\ 
Let $\chi$ be a compactly supported smooth even function on $\mathbb R$ such that  $\chi(x)=1$ if $|x|<1$, $\chi(x)=0$ if $|x|>2$ and $|\chi(x)|\leqslant 1$ for all $x\in \mathbb R$. Write $\chi_s(x)=\chi(x/s)$.

Let $G$ be the normalizing function given in line $\eqref{eq:G(x)}$. The derivative of $G$ is $G'(x) = \frac{2}{\sqrt{\pi}} e^{-x^2}$. Set $G_t(x)=G(tx)$. Recall that in Lemma \ref{lemma:G_tfinite} and Lemma \ref{lemma:G-sgn}, we have seen that
\begin{enumerate}
	\item $\forall t\geqslant 1$, $\|G(t\widetilde D)\|_{1,K}$ and $\|G(t\widetilde D_c)\|_{1,K}$ are uniformly bounded in $t$.
	\item $\|\sgn(\widetilde D_c)\|_{1,K}<\infty$.
	\item $	\lim_{t\to\infty}\|G_t(\widetilde D_c)-\sgn(\widetilde D_c)\|_{1, K}=0.$
\end{enumerate}
It follows that
\begin{equation}\label{eq:G^2-1}
\lim_{t\to\infty}\|G_t(\widetilde D_c)^2-1\|=0
\end{equation}
as $\sgn(\widetilde D_c)^2=1$.

We define
	\begin{equation}\label{eq:Fat}
	F_{a,t}(x)=\int_0^{xt} (G'* \widehat{\chi}_{a t})(y)dy-1,
	\end{equation} 
	where $a$ is a fixed positive number satisfying $a>\frac{4(K_\Gamma+K)}{\tau_\partial}$. 
	Note that the distributional Fourier transform of $F_{a, t}$ is supported on $[-2a t^2,2a t^2]$. 	
	
	Let $\Psi_T$ be the cut-off function on $\widetilde M_\infty$ which vanishes on $\partial\widetilde M\times [T,\infty)$ given in Definition \ref{def:cutoff}. Let  $ p_{F_{a,t}(\widetilde D)}$ and $q_{F_{a,t}(\widetilde D),T}$ be the elements as defined in line $\eqref{eq:idem}$ and $\eqref{eq:qFDT}$ respectively. We have 
	\begin{align*}
	&q_{F_{a,t}(\widetilde D),T}^2-q_{F_{a,t}(\widetilde D),T}=q_{F_{a,t}(\widetilde D),T}^2-q_{F_{a,t}(\widetilde D),T}-p_{F_{a,t}(\widetilde D)}^2+p_{F_{a,t}(\widetilde D)}\\
	=&-q_{F_{a,t}(\widetilde D),T}(p_{F_{a,t}(\widetilde D)}-q_{F_{a,t}(\widetilde D),T})-(p_{F_{a,t}(\widetilde D)}-q_{F_{a,t}(\widetilde D),T})p_{F_{a,t}(\widetilde D)} \\
	& \quad +(p_{F_{a,t}(\widetilde D)}-q_{F_{a,t}(\widetilde D),T})\\
	=&-q_{F_{a,t}(\widetilde D),T}\big(p_{F_{a,t}(\widetilde D)}-	\begin{psmallmatrix}
	1&0\\0&0
	\end{psmallmatrix}\big)(1-\Psi_T)\\
	&-\big(p_{F_{a,t}(\widetilde D)}-	\begin{psmallmatrix}
	1&0\\0&0
	\end{psmallmatrix}\big)(1-\Psi_T)\cdot p_{F_{a,t}(\widetilde D)}
	+\big(p_{F_{a,t}(\widetilde D)}-	\begin{psmallmatrix}
	1&0\\0&0
	\end{psmallmatrix}\big)(1-\Psi_T)
	\end{align*}

The propagations of $q_{F_{a,t}(\widetilde D),T}$ and $p_{F_{a,t}(\widetilde D)}$ are no more than $12a t^2$. Therefore, if we set $T=24a t^2$, then the operator 
$$q_{F_{a,t}(\widetilde D),{24a t^2}}^2-q_{F_{a,t}(\widetilde D),{24a t^2}}$$
on $\widetilde M_\infty$ coincides with the operator 
$$q_{F_{a,t}(\widetilde D_{\cyc}),{24a t^2}}^2-q_{F_{a,t}(\widetilde D_{\cyc}),{24a t^2}}$$
on $\partial\widetilde M\times\mathbb R$ under the canonical identification between the subspace $\partial\widetilde M\times\R_{\geqslant 0} $ of  $\partial\widetilde M\times\R$ and the cylindrical part $\partial\widetilde M\times\R_{\geqslant 0} $ of $\widetilde M_\infty$.

Note that $F_t-G_t$ is a Schwartz function and its Fourier transform is 
\begin{equation}
\widehat{ F}_t-\widehat{G}_t(\xi)=\widehat{G}_t(\xi)\cdot (1-\chi_{a t^2}(\xi))=\frac{1}{t}\widehat G(t^{-1}\xi)\big(1-\chi (\frac{\xi}{a t^2})\big).
\end{equation}
Since $\widehat G$ has Gaussian decay (away from the origin) with rate $1/4$, the function $\widehat{ F}_t-\widehat {G}_t$ has Gaussian decay on the whole real line with rate $1/4$. Therefore, by Lemma \ref{lemma:Fourier}, for any $\mu>1$,  there exists $C_1>0$ such that 
$$\|F_{a,t}(\widetilde D_{\cyc})_\gamma -G_t(\widetilde D_{\cyc})_\gamma\|_\op\leqslant C_1e^{-\frac{\tau_\partial^2\ell(\gamma)^2}{4\mu^2t^2}}$$
if $\ell(\gamma)$ is sufficiently large. Moreover, since the function $1-\chi(\frac{\xi}{a t^2})$ is supported on $\{\xi\in \mathbb R : |\xi|\geqslant a t^2 \}$, there exists $C_2>0$ such that
\begin{align}
\| F_{a,t}(\widetilde D_{\cyc})_\gamma-G_{t}(\widetilde D_{\cyc})_\gamma\|_\op&\leqslant \| F_{a,t}(\widetilde D_{\cyc})-G_t(\widetilde D_{\cyc})\|_\op \notag\\
&\leqslant \frac{1}{2\pi }\int_{|\xi|\geqslant a t^2}|\frac{1}{t}\widehat G(t^{-1} \xi)|d\xi
\leqslant C_2 e^{-\frac{a^2t^2}{4}} \label{eq:normdiff}
\end{align}

Similar to the proof of Lemma \ref{lemma:largespectralgap}, since we have $a>4(K_\Gamma+K)/\tau_\partial$, there exist positive constants $C$ and $\delta$ such that for any $t>0$,
\begin{equation}\label{eq:Ft-Gt}
\| F_{a,t}(\widetilde D_{\cyc})-G_t(\widetilde D_{\cyc})\|_{1,K}\leqslant C e^{-\delta t^2}.
\end{equation}
Therefore the norm $\| F_{a,t}(\widetilde D_{\cyc})\|_{1,K}$  is uniformly bounded for all $t\geqslant 1$. By the definition of $p_{F_{a,t}(\widetilde D_{\cyc})}$ (cf. line $\eqref{eq:idem}$), we see that the norm $\| p_{F_{a,t}(\widetilde D_{\cyc})}\|_{1, K}$ is also uniformly bounded for all $t\geqslant 1$, which in turn implies that the norm  $\|q_{F_{a,t}(\widetilde D_{\cyc}),T}\|_{1, K}$ is also uniformly bounded for all $t\geqslant 1$ and all $T\geqslant 0$, since the operator given by multiplication by $\Psi_T$ has zero propagation and  $\|\Psi_T\|_{op}\leqslant 1$ holds for all $T\geqslant 0$. 

Furthermore, we have 
\begin{align*}
	& \| F_{a,t}(\widetilde D_{\cyc})^2-1\|_{1, K}  \\
	& \leqslant \| F_{a,t}(\widetilde D_{\cyc})-G_t(\widetilde D_{\cyc})\|_{1, K}\cdot \| F_{a,t}(\widetilde D_{\cyc})+G_t(\widetilde D_{\cyc})\|_{1, K}+\|G_t(\widetilde D_{\cyc})^2-1\|_{1, K},
\end{align*}
where the right hand side goes to zero as $t$ goes to infinity by line $\eqref{eq:Ft-Gt}$ and $\eqref{eq:G^2-1}$. 
It follows that 
\begin{equation}\label{eq:converge}
\lim_{t\to\infty}\big\|p_{F_{a,t}(\widetilde D_{\cyc})}-	\begin{psmallmatrix}
	1&0\\0&0
\end{psmallmatrix}\big\|_{1, K}=0.
\end{equation}
On the other hand,  we have 
\begin{align*}
&q_{F_{a,t}(\widetilde D_{\cyc}),{24a t^2}}^2-q_{F_{a,t}(\widetilde D_{\cyc}),{24a t^2}}\\
=&-q_{F_{a,t}(\widetilde D_{\cyc}),{24a t^2}}\big(p_{F_{a,t}(\widetilde D_{\cyc})}-	\begin{psmallmatrix}
1&0\\0&0
\end{psmallmatrix}\big)(1-\Psi_{24a t^2})\\
&-\big(p_{F_{a,t}(\widetilde D_{\cyc})}-	\begin{psmallmatrix}
1&0\\0&0
\end{psmallmatrix}\big)(1-\Psi_{24a t^2})\cdot p_{F_{a,t}(\widetilde D_{\cyc})}\\
&+\big(p_{F_{a,t}(\widetilde D_{\cyc})}-	\begin{psmallmatrix}
1&0\\0&0
\end{psmallmatrix}\big)(1-\Psi_{24a t^2}).
\end{align*}
Applying line $\eqref{eq:converge}$, we have  
$$\lim_{t\to\infty}\|q_{F_{a,t}(\widetilde D_{\cyc}),{24a t^2}}^2-q_{F_{a,t}(\widetilde D_{\cyc}),{24a t^2}}\|_{1, K}=0.$$

Hence the proof is finished by setting $T = 24at^2$ and $F = F_{a, t}$ for some sufficiently large $t\gg 0$. 
\end{proof}

\subsection{Odd dimensional case}\label{subsec:oddl1}
Let us now turn to the proof of Theorem \ref{thm:main} for the case where $\dim M$ is odd.  Again, assume the same notation as in the Geometric setup $\ref{geoset}$. In particular, $\widetilde D$ is the Dirac operator on $\widetilde M_\infty$ and $\widetilde D_{\cyc}$ is the Dirac operator on the cylinder $\partial\widetilde M\times \R$. 

\begin{proof}[Proof of Theorem $\ref{thm:main}$ for the odd dimensional case]
	Similar to the construction of the higher index $\ind^\Gamma(D)\in K_1(C_r^\ast(\Gamma))$ as in Definition $\ref{def:oddindex}$. The proof of Theorem $\ref{thm:main}$ for the odd dimensional case can be reduced to the following claim, which is an $l^1$-analogue of Lemma \ref{lemma:fninvertible}. 
	\begin{claim}\label{prop:invertible}
		If $|\widetilde D_{\cyc}|>\frac{2(K_\Gamma+K)}{\tau_\partial}$, then there exist a positive number $T$, a positive integer $n$, and a smooth function $P$ satisfying condition (P1)-(P2) in Section \ref{subsec:oddindex}
		such that $\|1+f_n(\pm P(\widetilde D))\Psi_T\|_{1,K}<\infty$. Furthermore, we have 
		both
		$$\|(1+f_n(P(\widetilde D))\Psi_T)(1+f_n(-P(\widetilde D))\Psi_T)-1\|_{1,K}<1,$$
		and 
			$$\|(1+f_n(-P(\widetilde D))\Psi_T)(1+f_n(P(\widetilde D))\Psi_T)-1\|_{1,K}<1,$$
		where $\Psi_T$ is the cut-off function on $\widetilde M_\infty$ which vanishes on $\partial\widetilde M\times [T,\infty)$ given in Definition \ref{def:cutoff} and $f_n$  is the polynomial given in line \eqref{eq:f_n}.
	\end{claim}

Let us assume we have verified the claim for the moment.  We conclude that the element $1+f_n(P(\widetilde D))\Psi_T$  admits both left and right inverses in  $(\cL(\widetilde M_{T+nN_P})_K^\Gamma)^+$, hence is invertible in  $(\cL(\widetilde M_{T+nN_P})_K^\Gamma)^+$, where  $(\cL(\widetilde M_{T+nN_P})_K^\Gamma)^+$ the unitization of $\cL(\widetilde M_{T+nN_P})^\Gamma_K$ (cf. Definition $\ref{def:konenorm}$).  It follows that the higher index $\ind^\Gamma(\widetilde D)$ represented by this invertible element lies in $K_1(\cL(\widetilde M_{T+nN_P})^\Gamma_K) \cong K_1(l_K^1(\Gamma))$, which proves Theorem $\ref{thm:main}$ for the odd dimensional case.  

We remark that an $l^1$-analogue of the inequality from line \eqref{eq:inv} does \emph{not} imply that $(1+f_n(P(\widetilde D))\Psi_T)$ is invertible, since the weighted $l^1$-norm of $e^{2\pi i P(\widetilde D)}$ may be much greater than $1$. This is the reason for trying to prove the two inequalities in Claim \ref{prop:invertible} instead.
\end{proof}

Now let us prove Claim \ref{prop:invertible}, hence complete the proof of Theorem $\ref{thm:main}$ for the odd dimensional case.
\begin{proof}[Proof of Claim $\ref{prop:invertible}$] 
		Let  $ G_t$ and $F_{a,t}$ be the functions given in  line \eqref{eq:G(x)} and \eqref{eq:Fat} respectively. Set
$$ P_{a,t}=\frac{ F_{a,t}+1}{2}\text{ and } Q_t=\frac{G_t+1}{2},$$
where $a$ is a fixed positive number satisfying $a>\frac{4(K_\Gamma+K)}{\tau_\partial}$. 
By line \eqref{eq:Ft-Gt}, we see that there exists $A>0$ such that 
$\|P_{a,t}(\widetilde D_{\cyc})\|_{1,K}\leqslant A$ and $\|Q_{t}(\widetilde D_{\cyc})\|_{1,K}\leqslant A$ for all $t\geqslant 1$. 
It follows from line \eqref{eq:e^{}-fn} that
$$
\|e^{ 2\pi i P_{a,t}(\widetilde D_{\cyc})}\|_{1,K}\leqslant e^{2\pi A},~\|e^{ 2\pi i Q_{t}(\widetilde D_{\cyc})}\|_{1,K}\leqslant e^{2\pi A},
$$
$$
\|1+f_n( P_{a,t}(\widetilde D_{\cyc}))\|_{1,K}\leqslant e^{2\pi A}
+{e^{2\pi}}A
,
$$
and
\begin{equation}\label{4.16}
\|e^{ 2\pi i P_{a,t}(\widetilde D_{\cyc})}-\big(1+f_n(  P_{a,t}(\widetilde D_{\cyc}))\big)\|_{1,K}\leqslant \frac{e^{2\pi A}+e^{2\pi}A}{(n+1)!}.
\end{equation}

By the inequality from line $\eqref{eq:Ft-Gt}$, we see that 
\[ \| P_{a,t}(\widetilde D_{\cyc})-Q_t(\widetilde D_{\cyc})\|_{1,K} \to 0, \] as $t$ goes to infinity. Therefore we have
\begin{equation}\label{4.17}
\begin{split}
&\|e^{2\pi i P_{a,t}(\widetilde D_{\cyc})}-e^{2\pi iQ_t(\widetilde D_{\cyc})}\|_{1,K}\\
\leqslant& \|e^{2\pi iQ_t(\widetilde D_{\cyc})}\|_{1,K}\|e^{2\pi i( P_{a,t}(\widetilde D_{\cyc})-Q_t(\widetilde D_{\cyc}))}-1\|_{1,K}\\
\leqslant& e^{2\pi A}(e^{2\pi \| P_{a,t}(\widetilde D_{\cyc})-Q_t(\widetilde D_{\cyc})\|_{1,K}}-1),
\end{split}
\end{equation}
which goes to zero as $t$ goes to infinity.

Let $H$ be the Heaviside step function $H(x) = \frac{\sgn(x)+1}{2}$. Since $\widetilde D_{\cyc}$ is invertible, we have that $e^{2\pi iH(\widetilde D_{\cyc})}=1$. By Lemma \ref{lemma:G-sgn}, we have
\begin{equation}\label{4.18}
\begin{split}
\|e^{2\pi iQ_t(\widetilde D_{\cyc})}-1\|_{1,K}=&
\|e^{2\pi i(Q_t(\widetilde D_{\cyc})-H(\widetilde D_{\cyc}))}-1\|_{1,K}\\
\leqslant& e^{2\pi \|Q_t(\widetilde D_{\cyc})-H(\widetilde D_{\cyc})\|_{1,K}}-1.
\end{split}
\end{equation}
which goes to zero as $t$ goes to infinity.

By combining the inequalities from line \eqref{4.16}, \eqref{4.17} and \eqref{4.18}, we see there exist $t_\cyc>0$ and $n_\cyc>0$ such that if $n\geqslant n_\cyc$ and $t\geqslant t_\cyc$, then
\begin{equation}\label{eq:4.19}
(1+e^{2\pi A}+e^{2\pi }A)\|f_n(P_{a,t}(\widetilde D_{\cyc}))\|_{1,K}<1/6.
\end{equation}
The same inequality also holds if we replace $P_{a,t}(\widetilde D_{\cyc})$ by $-P_{a,t}(\widetilde D_{\cyc})$. Furthermore, since the operator given by the multiplication of $\Psi_T$ has zero propagation and  $\|\Psi_T\|_{op}\leqslant 1$ for all $T\geqslant 0$, it follows that 
\begin{equation}
	(1+e^{2\pi A}+e^{2\pi }A)\|f_n(P_{a,t}(\widetilde D_{\cyc})) \Psi_T\|_{1,K}<1/6.
\end{equation}
for all $T\geqslant 0$. We conclude that  
$$
\|(1+f_n(P_{a,t}(\widetilde D_\cyc)))(1+f_n(-P_{a,t}(\widetilde D_\cyc)))-1\|_{1,K}<1/3,
$$
and
$$
\|(1+f_n(P_{a,t}(\widetilde D_\cyc))\Psi_T)(1+f_n(-P_{a,t}(\widetilde D_\cyc))\Psi_T)-1\|_{1,K}<1/3,
$$
for all $n\geqslant n_\cyc$ and $t\geqslant t_\cyc$, which implies that 
\begin{equation}\label{eq:4.20}
\begin{split}
&\|\big(1+f_{n}(P_{a,t}(\widetilde D_{\cyc}))\big)\big(1+f_{n}(-P_{a,t}(\widetilde D_{\cyc}))\big)\\
&\quad -\big(1+f_{n}(P_{a,t}(\widetilde D_{\cyc}))\Psi_T\big)\big(1+f_{n}(-P_{a,t}(\widetilde D_{\cyc}))\Psi_T\big)\|_{1,K} <2/3.
\end{split}
\end{equation}
for all $n\geqslant n_\cyc$ and $t\geqslant t_\cyc$.

Now we turn to the Dirac operator $\widetilde D$ on $\widetilde M_\infty$. Note that $\widetilde D$ is not invertible in general, and the norm $\|P_{a,t}(\widetilde D)\|_{1,K}$ may not uniformly bounded for $t\geqslant 1$ in general. On the other hand, by construction,  $P_{a,t}(\widetilde D)$ has propagation $\leqslant 2at^2$, and $\|P_{a,t}(\widetilde D)\|_\op\leqslant 1$ for all $t>0$. Therefore by the definition of $\|\cdot\|_{1, K}$ in Definition \ref{def:konenorm}, there exist $C>0$ and $\lambda>0$ such that 
$$\|P_{a,t}(\widetilde D)\|_{1,K} = \sum_{\gamma\in \Gamma} e^{K\ell(\gamma)} \|P_{a,t}(\widetilde D)_\gamma\|_{op}\leqslant  Ce^{\lambda (K+K_\Gamma) at^2},$$
cf. line $\eqref{eq:quasi}$.
Let us denote  $A_t=Ce^{\lambda (K+K_\Gamma)at^2}$. We have 
$$
\|e^{ 2\pi i P_{a,t}(\widetilde D)}\|_{1,K}\leqslant e^{2\pi A_t},~
\|1+f_n( P_{a,t}(\widetilde D))\|_{1,K}\leqslant e^{2\pi A_t}
+{e^{2\pi}}A_t
,
$$
and
\begin{equation}\label{4.20}
\|e^{ 2\pi i P_{a,t}(\widetilde D)}-\big(1+f_n(  P_{a,t}(\widetilde D))\big)\|_{1,K}\leqslant \frac{e^{2\pi A_t}+e^{2\pi}A_t}{(n+1)!}.
\end{equation}
The same inequalities hold if we replace $P_{a,t}(\widetilde D)$ by $-P_{a,t}(\widetilde D)$.
Note that, for the positive numbers $n_c$ and $t_c$ from above,  there exists a positive integer $n_0$ such that $n_0\geqslant n_c$ and 
\begin{equation}
\frac{(e^{2\pi A_{t_\cyc}}+e^{2\pi}A_{t_\cyc})^2}{(n_0+1)!}<1/3.
\end{equation}
In particular, it follows that 
\begin{equation}\label{4.23}
\|(1+f_{n_0}(P_{a,t_\cyc}(\widetilde D)))(1+f_{n_0}(-P_{a,t_\cyc}(\widetilde D)))-1\|_{1,K}<1/3,
\end{equation}
since $e^{ 2\pi i P_{a,t_\cyc}(\widetilde D)}e^{-2\pi i P_{a,t_\cyc}(\widetilde D)}=1$.

Note that
\begin{align*}
&(1+f_{n_0}(P_{a,t_\cyc}(\widetilde D)))(1+f_{n_0}(-P_{a,t_\cyc}(\widetilde D)))\\
&-(1+f_{n_0}(P_{a,t_\cyc}(\widetilde D))\Psi_T)(1+f_{n_0}(-P_{a,t_\cyc}(\widetilde D))\Psi_T)\\
=&f_{n_0}(P_{a,t_\cyc}(\widetilde D))(1-\Psi_T)+f_{n_0}(-P_{a,t_\cyc}(\widetilde D))(1-\Psi_T)\\
&+f_{n_0}(P_{a,t_\cyc}(\widetilde D))(1-\Psi_T)f_{n_0}(-P_{a,t_\cyc}(\widetilde D))\Psi_T
+f_{n_0}(P_{a,t_\cyc}(\widetilde D))f_{n_0}(-P_{a,t_\cyc}(\widetilde D))(1-\Psi_T).
\end{align*}
Since the propagation of $f_{n_0}(\pm P_{a,t_\cyc}(\widetilde D))$ is no more than $2n_0a t_\cyc^2$, if we choose $T=4n_0a t_\cyc^2$,  then the operator 
\begin{align*}
&(1+f_{n_0}(P_{a,t_\cyc}(\widetilde D)))(1+f_{n_0}(-P_{a,t_\cyc}(\widetilde D)))\\
&-(1+f_{n_0}(P_{a,t_\cyc}(\widetilde D))\Psi_{4n_0a t^2_\cyc})(1+f_{n_0}(-P_{a,t_\cyc}(\widetilde D))\Psi_{4n_0a t_\cyc^2})
\end{align*}
on $\widetilde M_\infty$ is identified with
\begin{align*}
&(1+f_{n_0}(P_{a,t_\cyc}(\widetilde D_{\cyc})))(1+f_{n_0}(-P_{a,t_\cyc}(\widetilde D_{\cyc})))\\
&-(1+f_{n_0}(P_{a,t_\cyc}(\widetilde D_{\cyc}))\Psi_{4n_0a t_\cyc^2})(1+f_{n_0}(-P_{a,t_\cyc}(\widetilde D_{\cyc}))\Psi_{4n_0a t^2_\cyc})
\end{align*}
on $\partial \widetilde M\times\mathbb R$ under the canonical identification between the subspace $\partial\widetilde M\times\R_{\geqslant 0} $ of  $\partial\widetilde M\times\R$ and the cylindrical part $\partial\widetilde M\times\R_{\geqslant 0} $ of $\widetilde M_\infty$. Thus it follows from line \eqref{eq:4.20} that
\begin{equation}
\begin{split}
&\|\big(1+f_{n_0}(P_{a,t_\cyc}(\widetilde D))\big)\big(1+f_{n_0}(-P_{a,t_\cyc}(\widetilde D))\big)\\
&\quad -\big(1+f_{n_0}(P_{a,t_\cyc}(\widetilde D))\Psi_{4n_0a t^2_\cyc}\big)\big(1+f_{n_0}(-P_{a,t_\cyc}(\widetilde D))\Psi_{4n_0a t_\cyc^2}\big)\|_{1,K} < 2/3
\end{split}
\end{equation}
Together with line \eqref{4.23}, we have proved 
$$\|(1+f_n(P(\widetilde D))\Psi_T)(1+f_n(-P(\widetilde D))\Psi_T)-1\|_{1,K}<1,$$
by  setting  $P = P_{a,t_c}, n = n_0$ and $T = 4n_0a t_c^2$. The other inequality 
$$\|(1+f_n(-P(\widetilde D))\Psi_T)(1+f_n(P(\widetilde D))\Psi_T)-1\|_{1,K}<1,$$
can be proved in the exact same way in this case. This finishes the proof. 
\end{proof}

\section{An $l^1$-Atiyah-Patodi-Singer higher index formula}\label{sec:APS}
In this section, we prove an $l^1$-version of the Atiyah-Patodi-Singer higher index theorem.

Let us first recall the following definition of cyclic cocycles (cf. \cite[part II]{Connes}\cite[chapter 3, \S1]{ConnesNCG}) 

\begin{definition}\label{def:delocal}
Let $\Gamma$ be a discrete group. 
\begin{enumerate}[label=(\arabic*)]
	\item We say a map
	\[ \varphi\colon \Gamma^{n+1} = \underbrace{\Gamma \times \cdots \times \Gamma}_{(n+1) \textup{ copies} }\to\C \]  is a cyclic $n$-cocyle on $\Gamma$ if $\varphi$ is cyclic, that is, 
	$$\varphi(\gamma_n,\gamma_0,\cdots,\gamma_{n-1})=(-1)^n\varphi(\gamma_0,\gamma_1,\cdots,\gamma_n)
	,\ \forall\gamma_i\in\Gamma,
	$$ and $\varphi$ is closed, that is, 
	\begin{align*}
		(b\varphi)(\gamma_0,\cdots,\gamma_{n+1}):=
		&\sum_{i=0}^n(-1)^i\varphi(\cdots,\gamma_i\gamma_{i+1},\cdots)\\
		& \quad +(-1)^{n+1}\varphi(\gamma_{n+1}\gamma_0,\gamma_1,\cdots,\gamma_n)  =0,
	\end{align*}
	for $\forall\gamma_i\in\Gamma$. 
	\item A cyclic $n$-cocycle $\varphi$ of $\Gamma$ is said to be delocalized if 
	$$ \varphi(\gamma_0,\gamma_1,\cdots,\gamma_n)=0$$
	for all $\gamma_i$ such that $\gamma_0\gamma_1\cdots\gamma_n=I$ where $I$ is the identity element of $\Gamma$. 
\end{enumerate}
\end{definition} 
Let us fix a length function $\ell$ on $\Gamma$. We have the following notion of exponential growth for maps on $\Gamma^{n+1}$. 
\begin{definition}
A map $\varphi\colon \Gamma^{n+1} \to \mathbb C$
 is said to have at most exponential growth with respect to the length function $\ell$ on $\Gamma$ if there are positive numbers $C$ and $K_\varphi$ such that $$|\varphi(\gamma_0,\gamma_1,\cdots,\gamma_n)|\leqslant Ce^{K_\varphi (\ell(\gamma_0)+\ell(\gamma_1)+\cdots+\ell(\gamma_n))},\ \forall \gamma_i\in\Gamma.$$
\end{definition}

If $\varphi$ has exponential growth  rate $K_\varphi$, then it extends to a bounded linear functional on $l^1_{K_\varphi}(\Gamma)^{\otimes (n+1)}$, where the maximal tensor product is used. Let $\cS$ be the algebra of trace class operators on a Hilbert space. The map $\varphi$ induces the following bounded linear functional
$$\varphi\#\tr\colon \cS\otimes  l^1_{K_\varphi}(\Gamma)^{\otimes (n+1)}\to \C$$
by setting $$\varphi\#\tr(a)=\sum_{\gamma_0,\cdots,\gamma_n\in\Gamma }\tr(a_{\gamma_0,\cdots,\gamma_n})\varphi(\gamma_0,\cdots,\gamma_n),$$
for all 
$$ a=\sum_{\gamma_0,\cdots,\gamma_n\in\Gamma }a_{\gamma_0,\cdots,\gamma_n}\gamma_0\otimes\cdots\otimes \gamma_n\in \cS\otimes  l^1_{K_\varphi}(\Gamma)^{\otimes (n+1)}.$$

Now let us recall the definition of the delocalized higher eta invariant (cf. \cite[section 3]{CWXY}). 
Let $X$ be a complete Riemannian spin manifold. Let $\Gamma$ be a finitely represented group and $\widetilde X$ a normal $\Gamma$-cover of $X$.
\begin{definition}\label{def:h-del-eta}
  Suppose the Dirac operator $\widetilde D_X$ on $\widetilde X$ is invertible. 
  \begin{enumerate}
  	\item If $\dim X$ is odd and $\varphi$ is a delocalized cyclic $(2m)$-cocycle of $\Gamma$,  we define the delocalized higher eta invariant of  $\widetilde D_X$ at $\varphi$ to be
  	\begin{equation}\label{eq:eta_phi}
  		\eta_\varphi(\widetilde D_X)
  		\coloneqq \frac{m!}{\pi i}
  		\int_0^\infty \varphi\#\tr(\dot u_s u_s^{-1}\otimes\big( ((u_s-1)\otimes (u_s^{-1}-1)) )^{\otimes m}\big)ds,
  	\end{equation}
  	where $\{u_s\}_{s\in [0, \infty)}$ is the path of invertible elements given in line \eqref{eq:u_s}. 
  	\item If $\dim X$ is even and $\varphi$ is a delocalized cyclic $(2m+1)$-cocycle, we  define the delocalized higher eta invariant of  $\widetilde D_X$ at $\varphi$ to be
  	\begin{equation}
  		\eta_\varphi(\widetilde D_X)
  		\coloneqq \frac{(2m)!}{m!\pi i}
  		\int_0^\infty \varphi\#\tr([\dot p_s,p_s]\otimes(p_s-\begin{psmallmatrix}
  			1&0\\0&0
  		\end{psmallmatrix})^{\otimes (2m+1)})ds,
  	\end{equation}
  	where $p_s$ is given in line \eqref{eq:p_s}.
  \end{enumerate} 
\end{definition}

 In \cite[Section 3.3]{CWXY} the authors showed the integral of $\eta_\varphi(\widetilde D_X)$  converges absolutely, under the conditions that  the delocalized cyclic cocycle $\varphi$ has at most exponential growth with growth rate $K_\varphi$
and
$$|\widetilde D_X|>\frac{4(K_\Gamma+K_\varphi)}{\tau}.$$
In general, it is an open question whether the integral converges.


Recall that 	 $l^1_{K_\varphi}(\Gamma)\otimes\cS$ is a smooth dense subalgebra of $l^1_{K_\varphi}(\Gamma)\otimes\cK$, that is, $l^1_{K_\varphi}(\Gamma)\otimes\cS$ is dense and closed under holomorphic functional calculus in $l^1_{K_\varphi}(\Gamma)\otimes\cK$.  In particular, we have 
\[  K_\ast(l^1_{K_\varphi}(\Gamma)\otimes\cK) \cong K_\ast(l^1_{K_\varphi}(\Gamma)\otimes\cS). \] 
Denote by $(l^1_{K_\varphi}(\Gamma)\otimes\cS)^+$ the unitization of $l^1_{K_\varphi}(\Gamma)\otimes\cS$. For a delocalized $n$-cocycle $\varphi$ with exponential growth, we extend $\varphi\#\tr$ from $(l^1_{K_\varphi}(\Gamma)\otimes\cS)^{\otimes (n+1)}$ to $((l^1_{K_\varphi}(\Gamma)\otimes\cS)^+)^{\otimes (n+1)}$ by setting its value to be zero when encountering the extra identity.

It follows that for each cyclic $(2m)$-cocycle $\varphi$ with exponential growth rate $K_\varphi$, the Connes-Chern character map
\begin{equation}
	\ch_\varphi(p)\coloneqq \frac{(2m)!}{m!} \varphi\# \tr(p^{\otimes 2m+1})
\end{equation} 
for idempotents $p \in M_j(\C)\otimes (l^1_{K_\varphi}(\Gamma)\otimes\cS)^+$,  defines a homomorphism
$$\ch_\varphi\colon K_0(l^1_{K_\varphi}(\Gamma))\to\mathbb C.$$
Here $(l^1_{K_\varphi}(\Gamma)\otimes\cS)^+$ is the unitization of $l^1_{K_\varphi}(\Gamma)\otimes\cS$. If $\varphi$ is a cyclic $(2m+1)$-cocycle $\varphi$ with exponential growth rate $K_\varphi$, then the Connes-Chern character map
$$\ch_\varphi\colon K_1(l^1_{K_\varphi}(\Gamma)\otimes\cK)\to\mathbb C$$
is given by
\begin{equation}
\ch_\varphi(u)\coloneqq \frac{1}{m!} \varphi\# \tr((u^{-1}\otimes u)^{\otimes m}),
\end{equation} 
for invertible elements $u$ in $M_j(\C)\otimes (l^1_{K_\varphi}(\Gamma)\otimes\cS)^+$.

We have the following $l^1$-version of the Atiyah-Patodi-Singer higher index theorem.   
\begin{theorem}\label{thm:aps}
Assume the same notation from the geometric setup $\ref{geoset}$. Let $\varphi$ be the delocalized cyclic cocycle  that has exponential growth rate $K_\varphi$. If  the scalar curvature $k$ on $\partial M$ satisfies that
	$$\inf_{x\in\partial M}k(x)>\frac{16(K_\Gamma+K_\varphi)^2}{\tau^2_\partial},$$
	then we have
	$$\ch_\varphi(\ind^\Gamma_{1,K_\varphi}(D))=-\frac 1 2\eta_\varphi(\widetilde D_\partial).$$
\end{theorem}
\begin{proof}
	We prove the case where $\varphi$ is an even-degree cocycle and $\dim \partial M$ is odd. The case where $\varphi$ is an odd-degree cocycle and $\dim \partial M$ is even can be proved in the same way.
		
	As shown in \cite[section 6]{CWXY}, the formula in line \eqref{eq:eta_phi} extends to a homomorphism 
	$$\bar\eta_\varphi\colon K_1(\cL_{L,0}(\partial \widetilde M)^\Gamma_{K_\varphi})\to\mathbb C$$
	in the following sense(cf. \cite[Theorem 6.1$(b)$]{CWXY})
	\begin{equation}\label{eq:eta-det}
\eta_\varphi(\widetilde D_\partial)=\bar\eta_\varphi(\rho_{1,K_\varphi}(g_\partial)),
	\end{equation}
	where $\rho_{1,K_\varphi}(g_\partial)$ is the $l^1$-higher rho invariant of $(\partial M,g_\partial)$ given in Definition \ref{def:higherrho}. More precisely, the same proof of \cite[Theorem 6.1$(b)$]{CWXY} can be used to prove the equality $\eqref{eq:eta-det}$ above.

	Consider the following short exact sequence
	$$0\longrightarrow \cL_{L,0}(\partial\widetilde M)^\Gamma_{K_\varphi} \longrightarrow \cL_L(\partial\widetilde M)^\Gamma_{K_\varphi}\longrightarrow \cL(\partial\widetilde M)^\Gamma_{K_\varphi}\longrightarrow 0.$$
	We denote the  boundary map in the associated $K$-theory long exact sequence by
	$$\partial\colon K_*(\cL(\partial\widetilde M)^\Gamma_{K_\varphi})\cong K_*(l^1_{K_\varphi}(\Gamma))\longrightarrow K_{*+1}(\cL_{L,0}(\partial\widetilde M)^\Gamma_{K_\varphi}) .$$
	
	By \cite[Proposition 7.2]{CWXY}, we have that
	\begin{equation}
	\ch_\varphi(\ind^\Gamma_{1,K_\varphi}( D))=-\frac{1}{2}\bar\eta_\varphi\big(\partial( \ind^\Gamma_{1,K_\varphi}(D)) \big).
	\end{equation}
	
Let $Y$ be a subspace of $\sC$. We define $\cL_{L,0}(Y,\sC)^\Gamma_{K_\varphi}$ to be the completion of the collection of functions $f\colon[0,\infty)\to \C(\sC)^\Gamma_{K_\varphi}$ satisfying
\begin{enumerate}
	\item $\|f(s)\|_{1, K_\varphi}$ is uniformly bounded and uniformly continuous in $s$,
	\item the propagation of $f(s)$ goes to zero as $s$ goes to infinity,
	\item $f(0)=0$,
	\item and there exists $\lambda >0$ such that for any $s\geqslant 0$, $f(s)\cdot \chi_{Y, \lambda} = 0$, where $\chi_{Y, \lambda}$ is the characteristic function on the $\lambda$-neighborhood of $Y$. 
\end{enumerate}
with respect to the norm $\|f\|_{1, K_\varphi}:=\sup_{s\geqslant 0}\|f(s)\|_{1, K_\varphi}.$ Note that for each subspace $Y$ of $\sC$, the Banach algebra  $\cL_{L,0}(Y,\sC)^\Gamma_{K_\varphi}$ is a closed two-sided  ideal of $\cL_{L,0}(\sC)^\Gamma_{K_\varphi}$. 

Observe that 
\[ \cB_{L,0}(\sC)^\Gamma_{K_\varphi} = \cB_{L,0}(\sC_+,\sC)^\Gamma_{K_\varphi} + \cB_{L,0}(\sC_+,\sC)^\Gamma_{K_\varphi} \] 
and 
\[ \cB_{L,0}(\partial\widetilde M,\sC)^\Gamma_{K_\varphi} = \cB_{L,0}(\sC_+,\sC)^\Gamma_{K_\varphi} \cap  \cB_{L,0}(\sC_+,\sC)^\Gamma_{K_\varphi}. \] 
Therefore, we have the following Mayer-Vietoris sequence in $K$-theory: 
\begin{equation}\label{cd:mv}
	\begin{gathered}
\adjustbox{scale=0.84}
{\begin{tikzcd}[column sep=2ex,row sep=3ex]
		K_0(\cB_{L,0}(\partial\widetilde M)^\Gamma_{K_\varphi}) \arrow{r} & 		\begin{matrix}
			K_0(\cL_{L,0}(\sC_+,\sC)^\Gamma_{K_\varphi})\\
			\oplus\\
			K_0(\cL_{L,0}(\sC_-,\sC)^\Gamma_{K_\varphi})
		\end{matrix} \arrow[r] & K_0(\cB_{L,0}(\sC)^\Gamma_{K_\varphi}) \arrow{d}{\partial_{MV}} \\
		K_1(\cB_{L,0}(\sC)^\Gamma_{K_\varphi}) \arrow{u}  & 		\begin{matrix}
			K_1(\cL_{L,0}(\sC_+,\sC)^\Gamma_{K_\varphi})\\
			\oplus\\
			K_1(\cL_{L,0}(\sC_-,\sC)^\Gamma_{K_\varphi})
		\end{matrix} \arrow[l]                             & K_1(\cB_{L,0}(\partial\widetilde M)^\Gamma_{K_\varphi}) \arrow[l]
	\end{tikzcd}
}
\end{gathered}
\end{equation}
where we have used the identification
\[ K_\ast(\cB_{L,0}(\partial\widetilde M,\sC)^\Gamma_{K_\varphi})  \cong K_\ast(\cB_{L,0}(\partial\widetilde M)^\Gamma_{K_\varphi}). \]
 
	By the standard constructions of the boundary maps $\partial$ and $\partial_{MV}$, we have 
		\begin{equation}
	\partial\big(\ind^\Gamma_{1,K_\varphi}(D)\big)=\partial_{MV}\big(\rho_{1,K_\varphi}(g_\partial +dt^2) \big)\in K_1(\cL_{L,0}(\partial\widetilde M)^\Gamma_{K_\varphi}),
	\end{equation}
	where $\rho_{1,K_\varphi}(g_\partial +dt^2)$ is the $l^1$-higher rho invariant on $\partial\widetilde M\times\R$ (cf. Definition \ref{def:higherrho}) and  $\partial_{MV}$ is the boundary map
	$$ \partial_{MV}\colon K_0(\cB_{L,0}(\sC)^\Gamma_{K_\varphi})\longrightarrow K_{1}(\cB_{L,0}(\partial\widetilde M)^\Gamma_{K_\varphi})$$
	from the Mayer-Vietoris sequence above.  By the standard construction of the boundary map $\partial_{MV}$, we also have the following commuting diagram
	$$\xymatrix{K_1(\cL_{L,0}(\sC)^\Gamma_{K_\varphi})\ar[r]^-{\partial_{MV}}& K_0(\cL_{L,0}(\partial\widetilde M)^\Gamma_{K_\varphi})\\
		K_0(\cL_{L,0}(\partial\widetilde M)^\Gamma_{K_\varphi})\ar[u]^{\otimes \ind_L(D_\R)}\ar@{=}[ur]
	}
	$$
	where $\ind_L(D_\R)$ is the local higher index of the Dirac operator $D_\R=\frac{1}{i}\frac{d}{dx}$ on $\R$  (cf. \cite{Yulocalization}). On the other hand, the product formula from Theorem \ref{thm:product} implies that 
	\begin{equation}\label{eq:product}
	\partial_{MV}\big(\rho_{1,K_\varphi}(g_\partial +dt^2)\big)=\rho_{1,K_\varphi}(g_\partial)\in K_1(\cL_{L,0}(\partial\widetilde M)^\Gamma_{K_\varphi}).
	\end{equation}
	Hence we have
	\begin{align*}
		\ch_\varphi(\ind^\Gamma_{1,K_\varphi}(D))&=-\frac{1}{2}\bar\eta_\varphi\big(\partial(\ind^\Gamma_{1,K_\varphi}(\widetilde D))\big)
	=-\frac{1}{2}\bar\eta_\varphi\big(\partial_{MV}\big(\rho_{1,K_\varphi}(g_\partial +dt^2)\big)\big)\\
	&=-\frac{1}{2}\bar\eta_\varphi(\rho_{1,K_\varphi}(g_\partial))=-\frac 1 2\eta_\varphi(\widetilde D_\partial)	
	\end{align*}
This finishes the proof. 
\end{proof}

\subsection{An alternative proof of Theorem \ref{thm:aps}}\label{sec:alter}
As pointed out in the introduction, Theorem \ref{thm:aps} is an strengthening of the higher Atiyah-Patodi-Singer index formula in  \cite[Theorem 3.36]{CLWY}. Since \cite[Theorem 3.36]{CLWY} was stated and proved using the language of $b$-calculus. In this subsection, for the convenience of the reader, we give an alternative proof of Theorem \ref{thm:aps} using the language of $b$-calculus.

We retain the same notation from the geometric setup $\ref{geoset}$. For brevity, let us assume  the dimension of $\partial M$ is odd. The other case is completely similar.  

	Let us briefly review Lott's noncommutative differential higher eta invariant. We shall follow  closely the notation in Lott's paper \cite{Lott1}.  For any $K>0$, let $l^1_{K}(\Gamma)$ be the weighted $l^1$-algebra of $\Gamma$ from Definition $\ref{def:konenorm}$. 
	The universal graded differential algebra of $l^1_{K}(\Gamma)$ is defined to be  
	\[ \Omega_\ast(l^1_{K}(\Gamma)) = \bigoplus_{j=0}^\infty \Omega_k(l^1_{K}(\Gamma)) \]
	where as a vector space, $\Omega_j(l^1_{K}(\Gamma)) = l^1_{K}(\Gamma) \otimes ( l^1_K(\Gamma)/\mathbb C)^{\otimes j}$. As $l^1_{K}(\Gamma)$ is a Banach algebra, we shall consider the Banach completion of $\Omega_\ast(l^1_{K}(\Gamma))$, which will still be denoted by $\Omega_\ast(l^1_{K}(\Gamma))$. 
	
	Let $S(\partial M)$ be the spinor bundle on $\partial M$. We denote the corresponding $l^1_{K}(\Gamma)$-vector bundle by $\mathfrak S = ( \partial \widetilde M\times_\Gamma l^1_{K_\varphi}(\Gamma))\otimes S(\partial M)$ and the space of smooth sections by $C^\infty (\partial M; \mathfrak S)$. Now suppose $\psi$ is a smooth function on $\partial\widetilde M$ with compact support  such that 
	\[ \sum_{\gamma\in\Gamma}\gamma\psi=1. \] Then we have a superconnection  $\nabla\colon C^\infty(\partial M; \mathfrak S)\to C^\infty(\partial M; \mathfrak S\otimes_{l^1_{K_\varphi}(\Gamma)} \Omega_1(l^1_{K_\varphi}(\Gamma)))$ given by 
	$$\nabla(f)=\sum_{\gamma\in \Gamma}(\psi \cdot  \gamma f)\otimes_{l^1_{K_\varphi}(\Gamma)} d\gamma.$$  See \cite{MR790678} for more details of the superconnection formalism. 
	\begin{definition}[{\cite[Section 4.4 \& 4.6]{Lott1}}]
		Lott's noncommutative differential higher eta invariant $\widetilde \eta(\widetilde D_\partial)$ is defined by the formula
		\[ \widetilde \eta(\widetilde D_\partial)=\int_0^\infty \STR(\widetilde D_\partial  \, e^{-(t\widetilde D_\partial +\nabla)^2})dt, \]
		where $\STR$ is the corresponding supertrace, cf. \cite[Proposition 22]{Lott1}. 
	\end{definition}
	
	Combining Quillen's superconnection formalism with Melrose's $b$-calculus formalism \cite{RM93}, for each $t>0$, one defines the $b$-Connes-Chern character of $\widetilde D$ on $\widetilde M_\infty$  to be
	\[ b\mhyphen\Ch_t(\widetilde D) =  b\mhyphen\STR(e^{-(t\widetilde D + \nabla )^2}) \in \Omega_*(l^1_{K}(\Gamma)),  \] 
	where $b\text{-}\STR$ is the corresponding $b$-supertrace in the $b$-calculus setting. See for example \cite{Leichtnam} for more details.

Recall that every cyclic $n$-cocycle $\varphi$ of $\Gamma$ can be decomposed as $\varphi=\varphi_e+\varphi_d$, where $\varphi_d$ is the delocalized part of $\varphi$, i.e.,  
$$\varphi_d(\gamma_0,\gamma_1,\cdots,\gamma_{n})=\begin{cases}
	\varphi(\gamma_0,\gamma_1,\cdots,\gamma_{n})& \textup{ if } \gamma_0\gamma_1\cdots \gamma_{n}\ne e,\\
	0& \text{ otherwise.}
\end{cases}$$ 
The following theorem is an strengthening of the higher Atiyah-Patodi-Singer index formula in  \cite[Theorem 3.36]{CLWY}.
	\begin{theorem}\label{b-aps}
		Assume that $\varphi$ is a cyclic cocycle of $\Gamma$ with exponential growth rate $K_\varphi$.  If  the scalar curvature $k$ on $\partial M$ satisfies that
		\begin{equation}\label{eq:scalbd}
			\inf_{x\in\partial M}k(x)>\frac{16(K_\Gamma+K_\varphi)^2}{\tau^2_\partial},
		\end{equation} then
			\begin{equation}\label{eq:b-APS}
\ch_\varphi(\ind^\Gamma_{1,K_\varphi}(D))= \big\langle\varphi_e,\int_M \hat A\wedge\omega\big\rangle - \frac{1}{2} \big\langle\varphi,\widetilde\eta(\widetilde D_\partial)\big\rangle,
			\end{equation}
		where $\hat A$ is the associated $\hat A$-form on  
		$M$ and $\omega$ is an element in $\Omega^\ast(M)\otimes \Omega_\ast(l^1_{K_\varphi}(\Gamma))$ \textup{(}cf. \cite[Theorem 13.6]{Leichtnam}\textup{)} and $\varphi_e$ \textup{(}resp. $\varphi_d$\textup{)} is the localized \textup{(}reps. delocalized\textup{)} part of $\varphi$.
		Consequently, if both $\Gamma$ and $\varphi$ have sub-exponential growth, then the equality in line $\eqref{eq:b-APS}$ holds as long as $\widetilde D_\partial$ is invertible. 
	\end{theorem}
\begin{proof}
	Under the assumptions of the theorem, we have the following equality 
	\begin{align}
		&b\text{-}\Ch_{t_1}(\widetilde D)-b\text{-}\Ch_{t_0}(\widetilde D)  \notag\\
		= &  - \frac{1}{2}\int_{t_0}^{t_1}\STR(\widetilde D_\partial e^{-(s\widetilde D_\partial +\nabla)^2})ds + d\int_{t_0}^{t_1} b\text{-}\STR(\widetilde D e^{-(\nabla +s\widetilde D)^2})ds \label{eq:bvar}
	\end{align}
	holds   in $\Omega_\ast(l^1_{K_{\varphi}}(\Gamma))$,  where $d\colon \Omega_\ast(l^1_{K_{\varphi}}(\Gamma)) \to \Omega_{\ast+1}(l^1_{K_{\varphi}}(\Gamma))$ is the differential on $\Omega_\ast(l^1_{K_{\varphi}}(\Gamma))$,  cf. \cite[Section 6]{EG93b}\cite[Proposition 14.2]{Leichtnam}. By pairing both sides of  $\eqref{eq:bvar}$ with $\varphi$, we have 
	\[ \big\langle \varphi, b\text{-}\Ch_{t_1}(\widetilde D) \big\rangle - \big\langle \varphi, b\text{-}\Ch_{t_0}(\widetilde D) \big\rangle
	=  - \frac{1}{2}\int_{t_0}^{t_1} \big\langle \varphi,  \STR(\widetilde D_\partial e^{-(s\widetilde D_\partial +\nabla)^2})\big \rangle ds. \]
	By \cite[theorem 3.36]{CLWY}, we have 
		\begin{equation}
		\lim_{t\to\infty}\big\langle \varphi, b\textup{-}\Ch_t(\widetilde D))\big\rangle =\big\langle\varphi_e,\int_W \hat A\wedge\omega\big\rangle -\frac{1}{2}\big\langle\varphi,\widetilde\eta(\widetilde D_\partial)\big\rangle.
	\end{equation}
Now by Theorem $\ref{thm:main}$, the higher index $\ind^\Gamma(\widetilde D)$ of $\widetilde D$ lies in $K_0(l^1_{K_\varphi}(\Gamma))$. By adapting the methods\footnote{In fact, the construction of the higher index $\ind^\Gamma(\widetilde D)$ in Section $\ref{sec:2}$ can be thought as the $K$-theoretic equivalent of taking $b$-trace. } from \cite[section 12 \& 14]{Leichtnam}, one shows that 
\[ 	\lim_{t\to\infty}\big\langle \varphi, b\textup{-}\Ch_t(\widetilde D))\big\rangle = \ch_\varphi(\ind^\Gamma(\widetilde D)).   \]
This finishes the theorem. 
\end{proof}
	
Let $\widetilde\eta(\widetilde D_\partial)$ be Lott's noncommutative differential higher eta invariant from above. Clearly, we have 
\[ \big\langle\varphi,\widetilde\eta(\widetilde D_\partial)\big\rangle =  \big\langle\varphi_e,\widetilde\eta(\widetilde D_\partial)\big\rangle + \big\langle\varphi_d,\widetilde\eta(\widetilde D_\partial)\big\rangle \]	
In the following, we shall prove under the same assumptions as in Theorem $\ref{b-aps}$ that 
\begin{equation}\label{eq:iden-eta}
\big\langle\varphi_d,\widetilde\eta(\widetilde D_\partial)\big\rangle =   \eta_{\varphi_d}(\widetilde D_\partial),
\end{equation}
	where $\eta_{\varphi_d}(\widetilde D_\partial)$ is the delocalized higher eta invariant of $\widetilde D_\partial$ at $\varphi_d$ given in Definition $\ref{def:h-del-eta}$. Consequently, for delocalized cyclic cocycles, Theorem $\ref{b-aps}$ recovers Theorem $\ref{thm:aps}$. 
	
The equality $\eqref{eq:iden-eta}$ for identifying two different formulas of the delocalized higher eta invariant was first established by the authors in  \cite[Section 8]{CWXY}, under the assumption that $\Gamma$ has polynomial growth and $\varphi$ is a delocalized cyclic cocycle with polynomial growth.\footnote{ Here under the condition that $\Gamma$ has polynomial growth, assuming $\varphi$ to have polynomial growth does not impose any serious restrictions. Indeed, if $\Gamma$ has polynomial growth, then every cyclic cohomology class of $\Gamma$ has a representative that has polynomial growth and furthermore any two such representatives differ by a cyclic coboundary with polynomial growth.   } It was stated as an open question whether the equality holds in general. The following proposition answers this question positively for all groups, but under the condition that the spectral gap of $\widetilde D_{\partial}$ is sufficiently large, i.e. the equality in line $\eqref{eq:scalbd}$ holds. 

In order to make the discussion more transparent and also consistent with that from \cite[Section 8]{CWXY}, we shall work with the periodic version of Lott's noncommutative differential higher eta invariant. 
\begin{definition}[{\cite[Section 4.4 \& 4.6]{Lott1}}]\label{def:Lotteta} For each $\beta>0$, we define $\widetilde \eta(\widetilde D_\partial, \beta)$ by the following integral
	\[ \widetilde  \eta(\widetilde D_\partial, \beta)= \beta^{1/2}\int_0^\infty \STR(\widetilde D_\partial e^{-\beta(t\widetilde D_\partial +\nabla)^2})dt.  \]
Then the periodic version of Lott's noncommutative differential higher eta invariant of $\widetilde D_\partial$ is defined to be 
$$\tilde\eta^{\per}(\widetilde D_\partial)=\int_0^\infty e^{-\beta}\tilde\eta(\widetilde D, \beta^2)d\beta.$$
\end{definition}

	In \cite[Section 8]{CWXY}, using the Laplace transform, the authors showed that  of $\langle \varphi_d,\widetilde \eta(\widetilde D_\partial)\rangle$ is \emph{formally} equal to the following integral 
\begin{equation}\label{eq:eta-v}
	\sqrt{\pi} \cdot \frac{m!}{\pi i}
	\int_0^\infty \varphi\#\tr\big(\dot v_s v_s^{-1}\otimes\big( (v_s-1)\otimes (v_s^{-1}-1) \big)^{\otimes m}\big)ds,
\end{equation}
where 
\begin{equation}
v_s=\begin{cases}
(\widetilde D+si)(\widetilde D-si)^{-1} & \textup{ if } s>0,\\ 
1 & \textup{ if } s=0.\\
\end{cases}
\end{equation}

Now we are ready to prove the equality $\eqref{eq:iden-eta}$. 
 
\begin{proposition}
	Under the same assumptions as in Theorem $\ref{b-aps}$, we have 
	\[ \big\langle\varphi_d,\widetilde\eta^\per(\widetilde D_\partial)\big\rangle =   \sqrt{\pi} \cdot \eta_{\varphi_d}(\widetilde D_\partial).\]
\end{proposition}	
\begin{proof}
	
By the estimates from Lemma \ref{lemma:local_appro}, we see that the integrand in line \eqref{eq:eta-v} is finite for each $t>0$ and the integral converges for large $t$ (cf. \cite[Lemma 3.27]{CWXY}).
In Lemma $\ref{lm:con-est}$, under the large scalar curvature assumption in  line $\eqref{eq:scalbd}$,  we prove that   $\{v_s\}_{s\in[0,+\infty)}$ defines  an invertible element in $(\cL_{L,0}(\partial\widetilde M)^\Gamma_{K_\varphi})^+$. Consequently, it follows that  the integrand  in line \eqref{eq:eta-v} is continuous for all $t\in [0, \infty)$, hence the integral also converges for small $t$.  This proves that $\langle \varphi_d,\widetilde \eta^\per(\widetilde D_\partial)\rangle$ is equal to the integral in line \eqref{eq:eta-v}.

It remains to verify that the integral in line \eqref{eq:eta-v} is equal to $ \eta_{\varphi_d}(\widetilde D_\partial)$ from Definition $\ref{def:h-del-eta}$. This can be done by applying a standard transgression formula as follows. Let $w_{\lambda, s}$, $\lambda\in [0, 1]$ and $s\in [0, \infty)$,   be the path of invertible elements connecting $v_s$ and $u_s$ from line $\eqref{eq:path}$. A straightforward calculation shows that 
\begin{equation*}
	\partial_s\left(\varphi\#\tr(w^{-1}\partial_\lambda w \otimes(w\otimes w^{-1})^{\otimes m})\right)=
	\partial_\lambda\left(\varphi\#\tr(w^{-1}\partial_s w\otimes(w\otimes w^{-1})^{\otimes m})\right).
\end{equation*}	
It follows that 
\begin{align}
	& \int_0^T  \varphi\#\tr(\dot{v} v^{-1}\otimes (v\otimes v^{-1})^{\otimes m})ds - \int_0^T  \varphi\#\tr(\dot{u}u^{-1}\otimes (u\otimes u^{-1})^{\otimes m})ds \notag\\
	& =
	\int_0^1 \widetilde{\varphi\#\tr}(w^{-1}\partial_\lambda w\otimes(w\otimes w^{-1})^{\otimes m})\Big|_{s=0}^{s=T} d\lambda \label{eq:trans}
\end{align}
By the $l^1_K$-norm estimates above (see also Lemma $\ref{lemma:local_appro}$),  we have 
$$\varphi\#\tr(w^{-1}\partial_\lambda w\otimes(w\otimes w^{-1})^{\otimes m})\to 0, $$
as $s\to \infty$. Also, note that $w_{\lambda, 0} \equiv 1$ for all $\lambda\in [0, 1]$. This implies that the right hand side of the equality $\eqref{eq:trans}$ goes to zero as $T$ goes to infinity. This proves that the integral in line \eqref{eq:eta-v} is equal to $ \eta_{\varphi_d}(\widetilde D_\partial)$ from Definition $\ref{def:h-del-eta}$, hence completes the proof.

\end{proof}

\end{document}